\documentclass[12pt]{amsart}
\usepackage[a4paper,margin=28mm]{geometry}
\usepackage[all,2cell]{xy} \UseAllTwocells
\usepackage{tikz}
\usepackage{rotating}
\usepackage{xspace}
\usepackage[pagewise]{lineno}
\usepackage[usenames,dvipsnames]{pstricks}
\usepackage{amssymb,mathrsfs,amsmath}
\usepackage[letterspace=-45]{microtype}

\usepackage{eucal}
\usepackage[usenames,dvipsnames]{pstricks}
\usepackage{epsfig}
\usepackage{pst-grad} 
\usepackage{pst-plot} 
\usepackage{tikz-cd}

\usetikzlibrary{arrows.meta}

\tikzcdset{
arrow style=tikz,
diagrams={>={Classical TikZ Rightarrow[length=2.0mm, width=1.25mm]}}
}

\usepackage[raiselinks=false,colorlinks=true,citecolor=blue,urlcolor=blue,linkcolor=blue,bookmarksopen=true]{hyperref}

\usepackage{xargs}

\usepackage{amsmath, amssymb}
\usepackage[shortlabels]{enumitem}

\catcode`\@=11
\def\@evenfoot{\rule{0pt}{20pt}[\today] \hfill [{\tt \jobname.tex}]}
\def\@oddfoot{\rule{0pt}{20pt}{[\tt \jobname.tex}]\hfill [\today]}
\catcode`\@=13
\newtheorem{theorem}{Theorem}[section]
\newtheorem{proposition}[theorem]{Proposition}
\newtheorem{lemma}[theorem]{Lemma}

\theoremstyle{definition}
\newtheorem{definition}[theorem]{Definition}
\newtheorem{example}[theorem]{Example}

\newtheorem{remark}[theorem]{Remark}
\newtheorem{observation}[theorem]{Observation}
\newtheorem{warning}[theorem]{Warning}
\newtheorem{apology}[theorem]{Apology}

\newtheorem{taller}[theorem]{$\!\!$}
\newenvironment{blanko}[1]%
{\begin{taller}{\normalfont\bfseries  #1}\normalfont}%
{\end{taller}}

\makeatletter

\newcommand{\unit}{\pi}

\def\Ob{{\rm Ob}}

\def\catO{{\mathbb O}}
\def\catP{{\mathbb P}}

\def\DDD{{\mathsf D}}

\def\Dt{\simplexcategory\kern-1pt^\ttt}

\newcommand{\KK}{\mathcal{K}}

\DeclareSymbolFont{bbold}{U}{bbold}{m}{n}
\DeclareSymbolFontAlphabet{\mathbbold}{bbold}
\def\id{\mathbbold{1}}

\def\Fin{{\mathbb{F}\mathrm{in}}}
\def\Fin{{\mathbb{F}}}

\def\fop{fibrewise order-preserving\xspace}

\def\ff{functorially factorisable\xspace}
\def\strictlyfactorisable{strictly factorisable\xspace}

\def\dd{\mathsf{d}}
\def\ss{\mathsf{s}}
\def\dw{\mathrm{\delta}}
\def\sw{\mathrm{\gamma}}
\def\rrr{\mathsf{r}}
\def\iii{\mathsf{i}}
\def\uuu{\mathsf{u}}

\def\tttt{\mathsf{t}}

\DeclareRobustCommand\upperstar{%
  \mathchoice%
    {\kern0pt\raise0.55ex\hbox{$\displaystyle *$}\kern0.8pt}
    {\kern0pt\raise0.58ex\hbox{$\textstyle *$}\kern0.8pt}
    {\kern0pt\raise0.45ex\hbox{$\scriptstyle *$}\kern0.4pt}
    {\kern0pt\raise0.4ex\hbox{$\scriptscriptstyle *$}\kern0.2pt}
}%
\DeclareRobustCommand\lowerstar{%
  \mathchoice%
    {\kern0pt\raise-0.65ex\hbox{$\displaystyle *$}\kern0.8pt}
    {\kern0pt\raise-0.68ex\hbox{$\textstyle *$}\kern0.8pt}
    {\kern0pt\raise-0.55ex\hbox{$\scriptstyle *$}\kern0.4pt}
    {\kern0pt\raise-0.5ex\hbox{$\scriptscriptstyle *$}\kern0.2pt}
}%

\newcommand{\commutes}{\text{\texttt{"}}}

\makeatletter
\newcommand{\xRightarrow}[2][]{\ext@arrow 0359\Rightarrowfill@{#1}{#2}}
\makeatother

\makeatletter
\newcommand{\xLeftarrow}[2][]{\ext@arrow 0359\Leftarrowfill@{#1}{#2}}
\makeatother

\newcommand{\dbto}{\Rightarrow}

\newcommand{\xto}{\xrightarrow}

\newcommand{\drpullback}{\arrow[phantom]{dr}[very near start,description]{\lrcorner}}

\usepackage[all]{xy}
\newcommand{\cd}[2][]{\vcenter{\hbox{\xymatrix#1{#2}}}}

\newcommand{\name}[1]{\ulcorner #1\urcorner}
\newcommand{\isopil}{\stackrel{\raisebox{0.1ex}[0ex][0ex]{\(\sim\)}}%
			{\raisebox{-0.15ex}[0.28ex]{\(\rightarrow\)}}}
\newcommand{\isleftadjointto}{\dashv}
\newcommand{\thg}{{\mathord{\text{--}}}}

\newcommand{\SSS}{\mathsf{S}}
\newcommand{\TTT}{\mathsf{T}}

\newcommand{\ttt}{\mathsf{t}}

\def\wnerve{W}
\def\fnerve{P}

\def\and{{\mbox { and }}}
\def\listtodoname{List of Todos}
\def\listoftodos{\@starttoc{tdo}\listtodoname}
\makeatother

\newcommand{\fib}{\varphi}

\newcommand{\Fun}{\operatorname{Fun}}

\newcommand{\Decbot}[1]{\operatorname{Dec}_\bot{}\kern-2pt{#1}}
\newcommand{\Dectop}[1]{\operatorname{Dec}_\top{}\kern-2pt{#1}}

\providecommand{\norm}[1]{\left| {#1}\right|}

\providecommand{\kat}[1]{\textnormal{\lsstyle{{\texttt{#1}}}}}

\newcommand{\Cat}{\kat{Cat}}

\DeclareMathAlphabet{\mathbbe}{U}{bbold}{m}{n}
\newcommand{\simplexcategory}{\mathbbe{\Delta}}

\newcommand{\op}{^{\text{{\rm{op}}}}}

\title{Pita factorisation in operadic categories}
\author[M.~Batanin, J.~Kock, M.~Weber]{Michael Batanin, Joachim Kock, Mark Weber}

\subjclass[2020]{18M60 (Primary), 18N50 (Secondary)}
\catcode`\@=11
\catcode`\@=13

\begin{document}

\maketitle

\begin{abstract}
  In strictly factorisable operadic categories, every morphism $f$ factors uniquely as
  $f=\eta_f \circ \pi_f$ where $\eta_f$ is order-preserving and $\pi_f$ is a
  quasi\-bijection that is order-preserving on the fibres of $\eta_f$. We call it the
  pita factorisation. In this paper we develop some general theory to compensate for
  the fact that generally pita factorisations do not form an orthogonal factorisation
  system. The main technical result states that a certain simplicial object in Cat,
  called the pita nerve, is oplax (rather than strict as it would be for an orthogonal
  factorisation system). The main application is the result that the so-called
  operadic nerve of any operadic category is coherent. This result is a key ingredient
  in the simplicial approach to operadic categories developed in the `main paper'
  \cite{Batanin-Kock-Weber:mainpaper}, which motivated the present paper. We also show
  that in the important case where quasibijections are invertible, the pita nerve is a
  decomposition space (a.k.a.~$2$-Segal space).
\end{abstract}

\tableofcontents

\setcounter{page}{1}

\setcounter{secnumdepth}{3}
\setcounter{tocdepth}{1}


\section{Introduction}

In combinatorics, it is a standard technique to deal with finite sets by
working only with standard finite ordinals, but with maps that are not
necessarily order-preserving. The following fact is a fundamental tool:

  Every map $f : m \to n$ factors {\em uniquely} as $f = \eta_f \circ \pi_f$

\noindent where $\eta_f$ is order-preserving and $\pi_f$ is a fibrewise
order-preserving bijection. 

The middle object $m \to m' \to n$ is the ordinal sum
$\sum_{i\in n} f^{-1}(i)$, which can be considered just a reordering of the elements in
$m$.
The following picture is very suggestive of the working of this 
{\em pita factorisation}.\footnote{The terminology is just derived from the Greek 
letters $\pi$ and $\eta$, which we had been using consistently. Unable to come 
up with a short descriptive 
name, we finally opted to officialise the name {\em pita factorisation}. (We
are aware that pita is not spelled with an $\eta$ in Greek.)} Verbally, make a cut 
after all the crossings of strands, and insert $m'$ at this 
level:

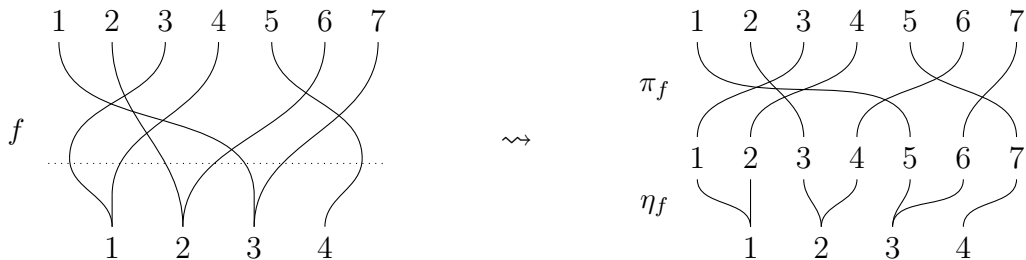
\[
\begin{tikzpicture}[scale=0.7]
 
\begin{scope}

  \begin{scope}[shift={(0,2.6)}]
	\draw (0.0,0) node (A1) {$1$};
	\draw (1.0,0) node (A2) {$2$};
	\draw (2.0,0) node (A3) {$3$};
	\draw (3.0,0) node (A4) {$4$};
	\draw (4.0,0) node (A5) {$5$};
	\draw (5.0,0) node (A6) {$6$};
	\draw (6.0,0) node (A7) {$7$};
  \end{scope}

  \begin{scope}[shift={(0,0)}]
	\coordinate (B1) at (0.2,0);
	\coordinate (B2) at (1,-0.7);
	\coordinate (B3) at (2,0);
	\coordinate (B4) at (3,0);
	\coordinate (B5) at (3.67,-0.7);
	\coordinate (B6) at (5,0);
	\coordinate (B7) at (5.7,0);
	\draw[dotted] (-0.2,-0.1) -- (6.1,-0.1);
  \end{scope}

  \draw (-0.8,0.5) node {\small $f$};
  
  \begin{scope}[shift={(0,-1.7)}]
	\draw (1.0,0) node (C1) {$1$};
	\draw (2.33,0) node (C2) {$2$};
	\draw (3.67,0) node (C3) {$3$};
	\draw (5,0) node (C4) {$4$};
  \end{scope}

  \draw (A1) to[out=-90, in=90] (B5);
  \draw (B5) to[out=-90, in=90] (C3);
  \draw (A2) to[out=-90, in=90] (C2);
  \draw (A3) to[out=-90, in=90] (B1);
  \draw (B1) to[out=-90, in=90] (C1);
  \draw (A4) to[out=-90, in=90] (B2);
  \draw (B2) to[out=-90, in=90] (C1);
  \draw (A5) to[out=-90, in=90] (B7);
  \draw (B7) to[out=-90, in=90] (C4);
  \draw (A6) to[out=-90, in=90] (C2);
  \draw (A7) to[out=-90, in=90] (C3);

\end{scope}

\draw (8.6,0.3) node {$\leadsto$};

\begin{scope}[shift={(12,0)}]

  \begin{scope}[shift={(0,2.6)}]
	\draw (0.0,0) node (A1) {$1$};
	\draw (1.0,0) node (A2) {$2$};
	\draw (2.0,0) node (A3) {$3$};
	\draw (3.0,0) node (A4) {$4$};
	\draw (4.0,0) node (A5) {$5$};
	\draw (5.0,0) node (A6) {$6$};
	\draw (6.0,0) node (A7) {$7$};
  \end{scope}

  \begin{scope}[shift={(0,0)}]
	\draw (0.0,0) node (B1) {$1$};
	\draw (1.0,0) node (B2) {$2$};
	\draw (2.0,0) node (B3) {$3$};
	\draw (3.0,0) node (B4) {$4$};
	\draw (4.0,0) node (B5) {$5$};
	\draw (5.0,0) node (B6) {$6$};
	\draw (6.0,0) node (B7) {$7$};
  \end{scope}

  \draw (A1) to[out=-90, in=90] (B5);
  \draw (A2) to[out=-90, in=90] (B3);
  \draw (A3) to[out=-90, in=90] (B1);
  \draw (A4) to[out=-90, in=90] (B2);
  \draw (A5) to[out=-90, in=90] (B7);
  \draw (A6) to[out=-90, in=90] (B4);
  \draw (A7) to[out=-90, in=90] (B6);

  \begin{scope}[shift={(0,-1.7)}]
	\draw (1.0,0) node (C1) {$1$};
	\draw (2.33,0) node (C2) {$2$};
	\draw (3.67,0) node (C3) {$3$};
	\draw (5,0) node (C4) {$4$};
  \end{scope}

  \draw (B1) to[out=-90, in=90] (C1);
  \draw (B2) to[out=-90, in=90] (C1);
  \draw (B3) to[out=-90, in=90] (C2);
  \draw (B4) to[out=-90, in=90] (C2);
  \draw (B5) to[out=-90, in=90] (C3);
  \draw (B6) to[out=-90, in=90] (C3);
  \draw (B7) to[out=-90, in=90] (C4);

    \draw (-0.8,1.3) node {\small $\pi_f$};
  \draw (-0.8,-0.9) node {\small $\eta_f$};

\end{scope}

\end{tikzpicture}
\]

The factorisation is implicit in a many situations in combinatorics, in
particular in connection with partitions (surjections), where it is often
involved in intermediate steps of `standardisation'. The factorisation also
comes up in operad theory, where it is custom to label inputs of operations with
numbers (so as to be able to talk about substitution into input slot number $i$
(the circle-i approach to operads, sometimes called Markl operads)). This often
requires reindexing and keeping track of permutations. In some situations, the
combinatorics becomes so involved that it is necessary to be explicit about the
factorisations. This was the case in our earlier work 
(in Batanin~\cite[Prop.~3.1]{BMEH} where explicit use of the 
factorisations was required in the proof of a general 
higher Eckmann--Hilton argument; in Weber~\cite{Weber:1503.07585} where 
these factorisations
are involved to compute codescent objects of crossed internal categories; and 
in Batanin--Kock--Weber
\cite{Batanin-Kock-Weber:1510.08934}, where the factorisations enter the proof
of the biequivalence between regular patterns and substitudes and between
Feynman categories and coloured operads.)

Beyond the skeletal category of finite set, a version of pita
factorisation was established by Lavers~\cite{Lavers} in the category of
vines, whose morphisms are a mixture of set maps and braids. Lavers
showed that every morphism $f$ in this category factors as $f = \eta_f
\circ \pi_f$ where the fibrewise-order-preserving factor $\pi_f$ is now a
{\em braiding} rather than just a permutation~\cite[Prop.~8]{Lavers}.
This particular version of the pita factorisation was later exploited by Day
and Street~\cite{Day-Street:substitudes} who interpreted it as a
distributive law, and used it to generalise substitudes and braided
substitudes to certain lax symmetric or braided monoidal functors to
sylleptic (braided) Gray monoids. Weber~\cite{Weber:1503.07585} exploited
the braided form of the pita factorisation to prove that the category of
vines is the free braided monoidal category on a commutative monoid
object.

The theory of operadic categories was introduced and developed by Batanin
and Markl~\cite{Batanin-Markl:1404.3886}, originally for the purpose of
proving the duoidal Deligne conjecture. It is a general theory for
operad-like structures, where each operadic category has it own notion of
operads. The afore-mentioned skeletal category $\Fin$ of finite sets is
the operadic category whose operads are ordinary symmetric operads.
(It is important that the skeleton is used. The notion of
operadic category is not invariant under equivalence of categories.)
There are operadic categories for other kinds of operads, and these
operads in turn have their own notion of algebras, just like ordinary
symmetric operads.

Since operadic categories can be seen as generalisations of the skeletal
category $\Fin$ (and relate to it via the cardinality functor), and the
theory often resorts to arguments exploiting the pita factorisation, it
was natural to consider generalisations of the pita factorisation to
operadic categories,
cf.~\cite{Batanin-Markl:1812.02935,Batanin-Markl:2105.05198}. For many
important operadic categories $\catO$, such as categories of graphs or
trees, there is such a notion, whereby a general morphism $f:T\to S$ in
$\catO$ factors uniquely as $\pi_f$ followed by $\eta_f$, where $\eta_f$
is order-preserving (meaning that its cardinality is an order-preserving
map in $\Fin$) and $\pi_f$ is a quasibijection that is order-preserving
on fibres. A morphism in an operadic category is a quasibijection if all
fibres are trivial (in the technical sense, as recalled below in
\ref{qbij}). Such operadic categories are called \strictlyfactorisable
operadic categories. It is one of the conditions required in order to
have a general theory of Koszul duality in operadic categories
\cite{Batanin-Markl:1812.02935,Batanin-Markl:2105.05198}. In particular,
all the examples treated in \cite{Batanin-Markl:2105.05198} are
\strictlyfactorisable, including many variations of the category of
graphs, such as notably connected genus-graded graphs -- the operadic
category for which algebras over the terminal operad are precisely
modular operads~\cite[Prop.~5.11]{Batanin-Markl:2105.05198}.

In the present paper we develop further the theory of factorisable operadic
categories. The theory can be seen as an effort to learn to live with the fact
that the pita factorisation, useful and important as it is, is not an orthogonal
factorisation system! Indeed, the defining property of the 
factorisation is tied locally to the codomain -- there is no global 
characterisation of the left-hand class of maps. This local nature of 
the factorisations gives the theory a somewhat different flavour.

We are particularly interested in properties of certain associated
reflection functors, also an important aspect of orthogonal
factorisation systems
(cf.~Cassidy--H\'ebert--Kelly~\cite{Cassidy-Hebert-Kelly}).

For an orthogonal factorisation system, it is natural and useful to
construct a double category with one class of morphisms as horizontal
arrows and the other class vertical, so it can be seen as a simplicial
object in $\Cat$ of chains of maps in the second class -- the categories
are then with morphisms from the first class. (The double categories
obtained are so-called {\em vacant} double
categories~\cite{Mackenzie:1992}, \cite{Andruskiewitsch-Natale:0308228}; see
\cite{Stepan:2305.06714} and \cite{Juran:2501.01363}
for recent developments.) For pita factorisations, things are more
involved, because of the local nature of the factorisations. The natural
thing to do is to consider only locally order-preserving chains, meaning
chains $T_n \to \cdots \to T_0$ all of whose composites ending in $T_0$
are order-preserving. (Note that the `right-hand class' -- the
order-preserving maps -- does not have the usual closure property that
right-hand classes have in orthogonal factorisation systems: it is not
true (not even in $\Fin$) that if $gf$ and $g$ are order-preserving then
so is $f$.) The categories involved then have as maps certain fibrewise
order-preserving quasibijections. In its naive form, this is not a
simplicial object because the top face operators, those omitting $T_0$,
are missing, but this can be fixed by exploiting the reflection functor.
This provides top face operators, but they do not satisfy the simplicial
identities. Our main theorem states that this {\em pita nerve} is
nevertheless coherent in the appropriate sense:

\smallskip

\noindent {\bf Theorem A.} (Theorem~\ref{laxcoherence}.) {\em The pita nerve of a
strictly factorisable operadic category is a coherent toplax simplicial category.}

\smallskip

The notion of toplax simplicial category is a natural class of oplax
simplicial categories arising from decalage, as we explain. The
coherators are given by fibrewise order-preserving quasibijections, so
the pita nerve is almost a pseudo-simplicial groupoid. (We refrain from
calling it a quasi-pseudo-simplicial groupoid.)

In many cases of (strictly factorisable) operadic categories, all 
quasibijections are invertible. In this case the pita nerve takes values 
in groupoids instead of categories, and becomes a
pseudo-simplicial groupoid. In this case we show

\smallskip

\noindent {\bf Theorem B.} (Theorem~\ref{thm:decomp}.) {\em For a strictly
factorisable operadic category with invertible quasi\-bijections, the pita nerve is a
decomposition space (a.k.a.~$2$-Segal space).}

\smallskip

It is beyond the scope of this paper to pursue the decomposition-space
aspect to any depth, but we calculate the incidence bialgebra of the
operadic category $\Fin_{\mathrm{surj}}$ of finite sets and surjections.
It is a bialgebra we have not seen before, closely related to the Fa\`a
di Bruno bialgebra.

\smallskip

It is a salient feature of the pita nerve that the top face operators are
of a different kind than the remaining face operators. This may look
artificial at first sight, but there are several significant instances of
this phenomenon elsewhere. We point out a few:

(1) {\em Waldhausen's S-dot construction}: in it most naive form, the
S-dot construction on an abelian category takes $S_n$ to be the groupoid
of chains of $n-1$ composable epimorphisms (fibrations). The face
operators $d_0$ and $d_1$ from $S_2$ to $S_1$ then send an epimorphism to
its codomain and its domain, respectively. The special top face operator
$d_2$ returns the kernel. This gives only a pseudo-simplicial groupoid,
and one of Waldhausen's insights was to rectify this with a less naive
construction.

(2) {\em The two-sided bar construction of an operad} (see for
example~\cite{Kock-Weber:1609.03276}): here the top face operator is a
Kleisli map (with respect to the free-symmetric-strict-monoidal-category 
monad), whereas all the other face operators are ordinary maps. In
terms of trees this is because composing out levels in a tree (or deleting
the leaf level) produces again a tree, whereas deleting the root level
produces instead a forest.

(3) {\em Hereditary species}: Carlier~\cite{Carlier:1903.07964}
constructed an operadic category from any hereditary species in the sense
of Schmitt~\cite{Schmitt:hacs} via a simplicial groupoid (which is a
decomposition space). This has the flavour of the two preceding examples.
For instance, take the hereditary species of graphs: the simplicial
groupoid then has contractions of graphs in simplicial degree $2$. The
face operators $d_0$ and $d_1$ assign to such a contraction its codomain
and domain, respectively, whereas the special top face operator $d_2$
assigns the set of fibres (so it is a Kleisli map as in example (2), and
it gives only a pseudo-simplicial groupoid as in example (1)). (Cebrian
and Forero~\cite{Cebrian-Forero:2211.07721} generalised Carlier's
construction to so-called directed hereditary
species~\cite{Galvez-Kock-Tonks:1708.02570} and used our results
from~\cite{Batanin-Kock-Weber:mainpaper} to give a much simpler proof
(covering also Carlier's result).)

\smallskip

Although we develop general theory, the motivation is one particular
application of Theorem A. In the `main paper'
\cite{Batanin-Kock-Weber:mainpaper}, we develop a new simplicial approach
to operadic categories, but since one technical point turned out to be
quite involved and required theory of independent interest, it has now
become this stand-alone paper.
We briefly explain the relationship, referring to the paper
\cite{Batanin-Kock-Weber:mainpaper} for details. Work of Garner, Kock,
and Weber~\cite{Garner-Kock-Weber:1812.01750} suggested that the
structure of operadic category on a category $\catO$ should amount to a
kind of `undecking': a certain simplicial object $X$ of which $\catO$ is
the upper decalage -- very roughly ``$\Dectop{} X = N \catO$''. This
equation is true for unary operadic categories (categories for which all
objects have cardinality $1$) \cite{Garner-Kock-Weber:1812.01750}, but
for general operadic categories it is considerably more involved to make
sense of the idea. Garner, Kock, and Weber gave one solution, using a
modification of the decalage comonad, but the undecking idea is not the
most prominent in their work. In \cite{Batanin-Kock-Weber:mainpaper} we
work out a different implementation of the undecking idea, which seems
rather natural: the undecking now takes place in the Kleisli category
for the free-strict-symmetric-monoidal-category monad $\SSS$, and the
simplicial object $X$ is thus a simplicial groupoid rather than a
simplicial set. The equation is rather $\Dectop{} X = \SSS N \catO$. The
top degeneracy operators express the structure of chosen local terminals
(as discovered in \cite{Garner-Kock-Weber:1812.01750}), whereas the top
face operators account for the fibre functor. With some further minor
adjustments, we can actually use this equation to {\em define} operadic
categories, and all the axioms then end up as simplicial identities. It must
be stressed that it is necessary to let $X$ be pseudo-simplicial instead
of strict simplicial. A key point in the theory is this:

\smallskip

\noindent
{\bf Theorem}~\cite{Batanin-Kock-Weber:mainpaper}. {\em  For any operadic category
$\catO$, its operadic nerve $X$ (undecking) is a coherent pseudo-simplicial
groupoid.}

\smallskip 

We prove this by reducing to Theorem A of the present paper. The
non-strictness here is significant, and is not a shortcoming of the
formalism: it is closely related to the operadic-category axiom saying
that fibres of fibres are fibres. In the way this axiom is formulated
originally, one does not see any non-strictness, but if one attempts at
assembling the fibres into a list, one will note that the list of fibres
of a composite morphism is not equal to the concatenated list of fibres
of fibres -- a reordering is generally required, and this reordering is
an instance of pita factorisation. (Only operadic categories that are
actually over $\simplexcategory$ will have a strict undecking.) The {\em
coherence} of this nonstrictness ends up being one of the axioms for the
cardinality functor of an operadic category.

\section{Pita factorisation in the skeletal category $\Fin$}

We denote by $\Fin$ the skeletal category of finite sets, in which 
the objects are the finite ordinals $\underline{n} = \{1,2,\ldots,n\}$ and 
the arrows are arbitrary set maps (not necessarily order-preserving).
Note that $\Fin$ has a \emph{unique} terminal
object $\underline 1$, which we use to endow $\Fin$ with local terminal
objects.

Given $f : \underline m \rightarrow \underline n$ in $\Fin$ and $i \in
\underline n$, the inverse image $\{\,j \in \underline m \mid f(j) =
i\,\}$ is not in general an object in the skeletal category $\Fin$, but
as a subset of $\underline m$ it inherits a linear order, and there is a
unique order-preserving bijection with an object in $\Fin$, which we
denote $f^{-1}(i)$ and call the \emph{fibre of $f$ at $i$}. We denote by
\begin{equation}\label{eq:10}
  \epsilon_{f, i} : f^{-1}(i) \rightarrow \underline m
\end{equation}
the unique order-preserving injection whose image is $\{\,j \in \underline m \mid
f(j) = i\,\}$. This is the pullback
\[
\begin{tikzcd}
f^{-1}(i) \drpullback \ar[d] \ar[r, "\epsilon_{f,i}"] & \underline m \ar[d, "f"]  \\
\underline 1 \ar[r, "\name{i}"'] & \underline n
\end{tikzcd}
\]
in the category $\Fin$.

Given also
$g \colon \underline \ell \rightarrow \underline m$ in $\Fin$, we
write $g_{i}$ for the unique map of $\Fin$ rendering
\[
\begin{tikzcd}
(fg)^{-1}(i) \drpullback \ar[d, dotted, "g_i"'] \ar[r, "\epsilon_{fg,i}"] & \underline \ell \ar[d, "g"]  \\
f^{-1}(i) \drpullback \ar[d] \ar[r, "\epsilon_{f,i}"] & \underline m \ar[d, "f"]  \\
\underline 1 \ar[r, "\name{i}"'] & \underline n
\end{tikzcd}
\]
commutative, and call it the \emph{fibre map of $g$ with respect to $f$
at $i$}.

The order-preserving injections
$\epsilon_{f, i} : f^{-1}(i) \rightarrow \underline m$
assemble into a single bijection
$$
\pi_f : \underline m \isopil \sum_i f^{-1}(i)
$$
expressing $\underline m$ as a sum of its fibres. This provides
altogether the {\em pita factorisation}: {\em every arrow $f: \underline
m \to \underline n$ factors uniquely as
\[
\begin{tikzcd}
\underline m \ar[dd, "f"'] \ar[rd, "\pi_f"] &   \\
 & \underline m' \ar[ld, "\eta_f"] &[-25pt] = \sum_i f^{-1}(i)  \\
 \underline n &
\end{tikzcd}
\]
where $\eta_f$ is order-preserving and $\pi_f$ is bijective and fibrewise order-preserving.}
Note that here $\underline m'$ is the same object as $\underline m$ (since $\Fin$ is 
skeletal), but it is practical to keep it distinct in the notation to remember that the
natural identification with $\underline m$ is given by $\pi_f$, not by the identity map.

We refer to the picture in the introduction.

The pita factorisation does not constitute a factorisation system, 
because the left-hand component is not 
described intrinsically (the condition `order-preserving on fibres' depends 
on the original map), but it nevertheless has many nice properties, and 
can be made functorial. It is an important technical aspect of the theory
of operadic categories and the main topic of this paper.

\section{Operadic categories}

To fix notation, we briefly recall the definition of operadic category of
Batanin and Markl~\cite{Batanin-Markl:1404.3886}. The notion can be seen
as specifying a category with \emph{formal} notions of fibre and fibre
map. The fibres of a map need not be subobjects of the domain as in the
case of $\Fin$, but the axioms ensure that they retain many important
properties of fibres in $\Fin$.

\begin{definition}
  \label{def:operadic-category}\cite{Batanin-Markl:1404.3886}
  An \emph{operadic category} is given by the following data:
  \begin{enumerate}[label=(D\arabic*)]
  \item \label{data:operadic-Q1} A category $\catO$ endowed with chosen local
    terminal objects; we write $\tau_R:R \to U_R$ for the unique 
    morphism to the chosen terminal object of the component that $R$ belongs 
    to.
    \item \label{data:operadic-Q2} A \emph{cardinality functor}
      $\norm{\thg} \colon \catO \to \Fin$;
    \item \label{data:operadic-Q3} For each object $R\in \catO$ and each
      $i \in \norm{R}$ a \emph{fibre functor}
      $$\fib_{R,i} \colon \catO/R \to \catO$$
      whose action on objects and morphisms we denote as follows:
      \begin{align*}
        \cd[@C1.3em]{
          S \ar[rr]^-{g} && R
        } \qquad &\mapsto \qquad g^{-1}(i)\\
        \cd[@C1em@R-0.7em]{{T} \ar[rr]^-{f} \ar[dr]_-{gf} & &
          {S} \ar[dl]^-{g} \\ &
          {R}} \qquad &\mapsto \qquad
        f_{i} \colon (gf)^{-1}(i) \to g^{-1}(i)\rlap{ ,}
      \end{align*}
      referring to the object $g^{-1}(i)$ as the \emph{fibre
      of $g$ at $i$}, and the morphism
    $f_{i} \colon (gf)^{-1}(i) \linebreak \to g^{-1}(i)$ as the \emph{fibre
      map of $f$ with respect to $g$ at~$i$};
  \end{enumerate}
  all subject to the following axioms, where
  in~\ref{axQ:BM-fibres-of-local-fibres}, we write $\epsilon j$
  for the image of $j \in {\norm g}^{-1}(i)$ under the map
  $\epsilon_{\norm g, i} \colon {\norm g}^{-1}(i) \rightarrow \norm S$
  of~\eqref{eq:10}:

  \begin{enumerate}[label=(A\arabic*)]
  \item \label{axQ:BM-abs(lt)} If $R$ is a local terminal then
    $\norm{R}=\underline 1$;
  \item \label{axQ:BM-fibres-of-identities} For all $R \in \catO$ and
    $i \in \norm R$, the object $(\id_R)^{-1}(i)$ is chosen local terminal;
  \item \label{axQ:BM-67} For all $g \in \catO / R$ and $i \in \norm R$,
    one has $\norm{\smash{g^{-1}(i)}} = {\norm g}^{-1}(i)$, while for
    all $f \colon gf \rightarrow g$ in $\catO / R$ and $i \in \norm R$,
    one has $\norm{\smash{f_{i}}} = \smash{\norm{f}_{i}}$;
  \item \label{axQ:BM-fibres-of-tau-maps} For $R \in \catO$ and for the unique 
	element $1\in \norm{U_R}$, one has
    $\tau_R^{-1}(1) = R$, and for $g \colon S \to R$, one has
    $g_1 = g$;
  \item \label{axQ:BM-fibres-of-local-fibres} For
    $f \colon gf \rightarrow g$ in $\catO/R$, $i \in \norm R$ and
    $j \in \norm{g}^{-1}(i)$, one has that
    $(f_i)^{-1}(j) = f^{-1}(\epsilon j)$, and given also
    $h \colon gfh \rightarrow gf$ in $\catO / R$, one has
    $(h_i)_{j} = h_{\epsilon j}$.
  \end{enumerate}
\end{definition}

The preceding definitions are exactly those
of~\cite{Batanin-Markl:1404.3886} with only some slight repackaging changes
as in \cite{Garner-Kock-Weber:1812.01750}. For subtle variations on the axioms, and 
slightly different notions of operadic category,
see Lack~\cite{Lack:1610.06282} and Markl~\cite{Markl:2502.09163}.

The category $\Fin$ has a unique (and hence chosen) terminal object.
Axiom~\ref{axQ:BM-abs(lt)} now says that the cardinality functor 
preserves chosen local terminals.
Any slice category has a canonical choice of terminal object, 
namely the identity arrow. With respect to this choice, Axiom~\ref{axQ:BM-fibres-of-identities}
says that the fibre functor preserves chosen local terminals.
Axiom~\ref{axQ:BM-67} says that 
for $R \in \catO$ and $i \in \norm R$, the square 
\begin{equation}
\label{diag:fibres-and-cardinalities}
\begin{aligned}{ 
\xygraph{!{0;(2,0):(0,.5)::}
{\catO_{\smash{/R}}}="p0" [r] {\catO}="p1" [d] {\Fin}="p2" [l] {\Fin_{\smash{/|R|}}}="p3"
"p0" :"p1"^-{\fib_{R,i}} :"p2"^-{\norm{\thg}} 
:@{<-}"p3"^-{\fib_{|R|,i}} :@{<-}"p0"^-{\norm{\thg}_{/R}}}}
\end{aligned}
\end{equation}
commutes.
Axiom~\ref{axQ:BM-fibres-of-tau-maps} says that $\fib_{R,i}$ is the
domain functor whenever $R \in \catO$ is chosen local terminal.
Intuitively, the first clause of \ref{axQ:BM-fibres-of-local-fibres}
identifies the fibres of the fibre maps of a map, with the fibres of
that map. 
The second part of \ref{axQ:BM-fibres-of-local-fibres} says that the fibre maps of
the fibre maps of a map are themselves fibre maps of that map. For the 
details of these interpretations, 
see~\cite{Garner-Kock-Weber:1812.01750}.

\section{Fibrewise order-preserving diagrams}

\begin{definition}[{\cite[p.8]{Batanin-Markl:1812.02935}}] \label{qbij}
  A morphism $\sigma:T\to S$ in an operadic category $\catO$ is called a
  {\em quasi\-bijection} if all its fibres are trivial (meaning that they are 
  chosen local terminal objects).
\end{definition}

\begin{remark}\label{qbij=bij}
  Any quasibijection induces a bijection of finite sets after application
  of the cardinality functor. In particular, in the operadic category
  $\Fin$ the quasibijections are already bijections -- in fact, in $\Fin$
  the quasibijections are precisely the bijections.
  In general, quasibijections do not need to be invertible, and not all
  invertible morphism are necessarily quasibijections. (But the identities 
  are always quasibijections, by 
  Axiom~\ref{axQ:BM-fibres-of-identities}.)
  
  A good example where not all invertible morphisms are quasibijections 
  is the category of bouquets~\cite{Batanin-Markl:1404.3886}. Here, for an isomorphism to be a quasibijection
  it must preserve the colourings of the leaves.
 
  Examples where the quasibijections are not invertible are categories
  of $n$-ordinals and $n$-trees \cite{BMEH},
  see~\cite{Batanin-Markl:1812.02935}. In these examples, it is in fact
  relevant to localise by formally inverting the quasibijections, which
  leads to the theory of locally constant $n$-operads and a generalised
  Baez--Dolan--Freudenthal stabilisation
  theorem~\cite{Batanin-White:2001.05432}.
  
  One convenient criterion for invertibility of quasibijections in an
  operadic category $\catO$ is the existence of a discrete operadic
  fibration (or opfibration) $\catO \to \catP$ to an operadic category
  $\catP$ where all quasibijections are invertible (see
  \cite[Table~1]{Batanin-Markl:1812.02935}).
\end{remark}

\begin{definition}\label{def:fop}
  A morphism $\lambda: T\to S$ in an operadic category $\catO$ is called
  {\em order-preserving} if its image $|\lambda|$ under the cardinality 
  functor is order-preserving  in $\Fin$.
  
  Let 
  \begin{equation}\label{fmonotone}
    \xymatrix@C = +3em{
    T \ar[d]_{f} \ar[r]^{\sigma}& T' \ar[d]^{g}
    \\
    S \ar[r]_{\tau} & S'
	}
  \end{equation}
  be a commutative diagram in $\catO$. We call $\sigma$ {\em fibrewise
  order-preserving} with respect to $\tau$ if $\sigma$ induces
  order-preserving morphisms on the fibres of $f\tau$ and $g$. We also
  call such a square a {\em \fop square}. In the particular case of a
  \fop square with $\tau =\id$, we will say that the resulting
  commutative triangle is a {\em \fop triangle}.
\end{definition}

\begin{observation}
  Every map to a local terminal is clearly order-preserving, and as a
  consequence every square~\eqref{fmonotone} with $g=\id$ is a \fop
  square. In particular, $\sigma$ is always \fop with respect to itself.
\end{observation}

\begin{proposition}\label{verticalfop}
  In an operadic category $\catO$, if in the diagram 
  \begin{equation}\label{fmonotongen}
    \xymatrix@C = +3em{
     T\ar@/^-3.3ex/[dd]_{gf} \ar[d]_{f} \ar[r]^{\sigma}& T \ar[d]^{a}\ar@/^3.3ex/[dd]^{ba} 
           \\
    S \ar[d]_{g}\ar[r]^{\omega} & S \ar[d]^{b}  
      \\ 
     R \ar[r]_{\tau} & R    
  } 
  \end{equation}
  $\omega$ and $\sigma$ are \fop with respect to $\tau$, then $\sigma$ is
  \fop with respect to $\omega$.
\end{proposition} 

\begin{proof}
  The property of a square being \fop depends only on the image of this
  square under the cardinality functor, so we reduce to the 
  corresponding statement in the category $\Fin$. Here it is an 
  elementary calculation.
\end{proof}

\begin{proposition}\label{3out2}
  For a commutative diagram
  \begin{equation} \label{hcompoffop}
    \xymatrix@C = +3em{
     T \ar[d]_{a} \ar[r]^{\sigma}
     & T' \ar[d]_{b} \ar[r]^{\omega}
     & T'' \ar[d]^c
      \\
    S \ar[r]_{\tau} & S' \ar[r]_{\lambda} & S''
  }    
  \end{equation}
  in an operadic category $\catO$, assume the
  square on the right is a \fop square. Then:
  \begin{enumerate}
	\item The square on the left is \fop provided the composite square is
	\fop and $|\omega_i|: |(c\omega)^{-1}(i)|\to|c^{-1}(i)|$ is an injection
	for all $i\in |S''|$.

	\item The composite square is \fop provided the square on the left is
	\fop and $|\lambda|:|S'|\to |S''|$ is a surjection.

	\item If $\sigma,\omega,\tau,\lambda$ are  
	quasibijections, then  
	the composite square is \fop if and only if the left-hand square is
	\fop.
  \end{enumerate}
\end{proposition} 

\begin{proof}
  We introduce the following notation for the composites indicated with
  dashed arrows in the following diagram:
  \begin{equation*}
    \xymatrix@C = +5em@R = +3.5em{
     T \ar[d]_{a} \ar[r]^{\sigma} \ar@{-->}[dr]_h \ar@{-->}[drr]^(0.6){f} 
     & T' \ar[d]_(0.7){b} \ar[r]^{\omega}   \ar@{-->}[dr]^{e}
     & T'' \ar[d]^c
      \\
    S \ar[r]_{\tau} & S' \ar[r]_{\lambda} & S''
  }    
  \end{equation*}
  Consider the commutative tetrahedron
  \[
  \xymatrix@C = +2.5em@R = +1em{ & T' \ar[dd]^(.3){e} \ar[dr]^b & 
  \\
  T  \ar[ur]^\sigma    \ar@{-}[r]^(.7){h}\ar[dr]_f & \ar[r] & S' \ar[dl]^\lambda
  \\
  &S''&
  }
  \]
  By Axiom~\ref{axQ:BM-fibres-of-local-fibres} we get an equality
  $$
  (\sigma_i)_j = \sigma_i:h^{-1}(j)\to b^{-1}(j)
  $$
  for any $j\in |S'|$ and $i = |\lambda|(j)$. Hence $\sigma$ is
  order-preserving on fibres over $S'$ between $h$ and $b$ if it is 
  order-preserving on fibres of $f$ and $e$ over $S''$. On the other hand, the
  commutative diagram
  \[
  \xymatrix@C = +2.5em@R = +1em{ & T' \ar[dd]^(.3){e} \ar[dr]^\omega & 
  \\
  T  \ar[ur]^\sigma    \ar@{-}[r]^(.7){\omega\sigma}\ar[dr]_f & \ar[r] & T'' \ar[dl]^c
  \\
  &S''&
  }
  \]
  induces a commutative triangle 
  \[
    \xymatrix@C = +1em@R = +1em{
  f^{-1}(i)     \ar[rr]^{\sigma_i} \ar[dr]_{(\omega\sigma)_i} & & e^{-1}(i) \ar[dl]^{\omega_i}
      \\
      &c^{-1}(i)&
    }
  \]
  Therefore, $\sigma_i$ (as fibre map over $S''$) is order-preserving
  provided $(\omega\sigma)_i$ and $\omega_i$ are order-preserving over
  $S''$ and $|\omega_i|$ is an injection. So we conclude that the left-hand
  commutative square in \eqref{hcompoffop} is \fop\hspace{-1mm}.

  If $|\lambda|:|S'|\to |S''|$ is a surjection and $\sigma_j : h^{-1}(j) \to b^{-1}(j)$
  is order-preserving for each $j\in |S'|$, then we can clearly reverse the argument
  to see that the maps
  $(\omega\sigma)_i : f^{-1}(i) \to c^{-1}(i)$ are order-preserving for all $i\in 
  |S''|$.
  
  The last point of the proposition is a direct consequence of the first two points.
\end{proof} 
 
\section{Functorially factorisable operadic categories}
\label{funfact}

The following definitions were introduced in
\cite{Batanin-Markl:1812.02935,Batanin-Markl:2105.05198}.
\begin{definition}[{\cite[Def.~2.8]{Batanin-Markl:1812.02935}}]
  An operadic category $\catO$ is called {\em factorisable} if for any
  $f:T\to S$ there is a factorisation
  \begin{equation} \label{unit}
    \xymatrix@C = +1em@R = +1em{
      T      \ar[rr]^{\pi(f)} \ar[dr]_f & & S \ar[dl]^{\eta(f)}
      \\
      &R&
    }
\end{equation} 
where $\pi(f)$ is a quasibijection and $\eta(f)$ is order-preserving.
We will call such a factorisation {\em
pita factorisation}.
\end{definition}

The following lemma is proved in \cite[Lemma 2.10]{Batanin-Markl:1812.02935},
but in a slightly different form. Let $U$ denote a chosen local 
terminal object. 
\begin{lemma}\label{qb=id}
  If in an operadic category $\catO$ every morphism $!:T\to U$ admits a
  unique pita factorisation, then the only order-preserving
  quasibijections in $\catO$ are the identities.
\end{lemma}

\begin{proof}
  Let $\sigma:T\to S$ be an order-preserving quasibijection. We then have
  two pita factorisations of $!:T\to U$: $$T\stackrel{\id}{\to} T\to
  U \ \ \text{and} \ \ T\stackrel{\sigma}{\to} S\to U.$$ From
  uniqueness of factorisation it follows that $\sigma = \id$.
\end{proof} 

\begin{definition}[{\cite[Def.~2.9]{Batanin-Markl:1812.02935}}]
  An operadic category $\catO$ is called {\em strongly factorisable} if for
  every $f$ there is a unique pita factorisation for which \eqref{unit}
  is a \fop triangle.
\end{definition}
\begin{lemma}[{\cite[Lemma 2.10]{Batanin-Markl:1812.02935}}]\label{sf in Sydney}
  In a strongly factorisable operadic category, every order-preserving
  quasibijection is an identity.
\end{lemma} 

\begin{remark}
  The original definition of strongly factorisable category given in
  \cite{Batanin-Markl:1812.02935} required that $\pi(f)$ induces the identities
  on fibres. In view of Lemma~\ref{qb=id} that condition is actually automatic,
  so that the present definition is equivalent to the
  original~\cite{Batanin-Markl:1812.02935}.
\end{remark}

Let $\catO$ be an operadic category. For a fixed object $R\in \catO$, let
$\catO/_{\mathsf{ord}} R$ be the
full subcategory of the slice category
$\catO/R$ whose objects are the order-preserving morphisms. 
Let
\begin{equation}\label{iii}
  \iii:\catO/_{\mathsf{ord}} R\to \catO/R
\end{equation}
be the inclusion functor.

\begin{definition}\label{funfactor}
  We call $\catO$ a {\em functorially factorisable} operadic category if for every
  $R\in \catO$ the inclusion functor $\iii:\catO/_{\mathsf{ord}} R\to \catO/R$
  has a left adjoint $\rrr: \catO/R \to \catO/_{\mathsf{ord}} R$ such that for
  any $f:T\to R$ the commutative diagram representing the unit $\unit$ of the
  adjunction
  \begin{equation} \label{unitfop}
    \xymatrix@C = +1em@R = +1em{
      T      \ar[rr]^{\unit_f} \ar[dr]_f & & S \ar[dl]^{\rrr(f)}
      \\
      &R&
    }
  \end{equation}
  is a pita factorisation.
\end{definition} 

Any \ff operadic category is obviously factorisable. Conversely, for an
arbitrary factorisable operadic category, a choice of pita factorisations
provides a map back on objects
\begin{eqnarray*}
  r:\Ob(\catO/R) & \longrightarrow & \Ob(\catO/_{\mathsf{ord}} R)•  \\
  f & \longmapsto & \eta(f)
\end{eqnarray*}
together with a `unit' morphism $\pi(f):f\to \eta(f)$.

\begin{proposition} \label{hitrayadiagramma}
  A factorisable operadic category $\catO$ is \ff if there are:
\begin{enumerate}
  \item  A choice of pita factorisation for each morphism $f:T\to S;$
  \item   For each  composable pair of morphisms
\begin{equation}\label{chainfg}
  T\stackrel{f}{\to} S \stackrel{g}{\to} R
\end{equation}
 a choice of  morphism $\eta(f/g)$ making the following diagram commute:
 \begin{equation}\label{etafunctorial1}
    \xymatrix@C = +1em@R = +1em{
    T \ar@/^-3.3ex/[ddrrr]_{gf} \ar[drr]^{\hskip1em\pi(gf)}     \ar[rrrrrr]^f& & & & & & S  \ar@/^3.3ex/[ddlll]^{g} \ar[dll]_{\hskip-0.7em\pi(g)} 
     \\
     & &T'      \ar@{-->}[rr]^{\eta({f}/{g})} \ar[dr]_(0.2){\eta(gf)} & & S' \ar[dl]^(0.2){\hskip-0.7em\eta(g)} & &
      \\ 
     & & &R& & &
    }
\end{equation}
\end{enumerate}
These choices must satisfy the relations
 \begin{enumerate}
\item $\eta(g/h)\eta(f/hg) = \eta(gf/h)$ for any composable triple of morphism 
$T \stackrel f \to S \stackrel g \to R \stackrel h \to Q$
 \item $\eta(\id/f) = \id$. 
\end{enumerate}

\end{proposition} 
\begin{proof}
  This is just an unpacking of Definition~\ref{funfactor}.  
\end{proof} 

\begin{warning}
  Given $T \stackrel{f}\to S \stackrel{g}\to R$, if we iterate the pita
  factorisation so as to build up the diagram
  \[
  \begin{tikzcd}[sep={36pt,between origins}]
  T \ar[dd, "f"'] \ar[rd, "\pi(f)"] & &  \\
   & \cdot \ar[ld, "\eta(f)"] \ar[rd, "\pi\big(\pi(g)\eta(f)\big)"] & \\
   S \ar[dd, "g"'] \ar[rd, "\pi(g)"] && \cdot \ar[ld, 
   "\eta\big(\pi(g)\eta(f)\big)"] \\
   & \cdot \ar[ld, "\eta(g)"] & \\
   R & &
  \end{tikzcd}
  \]
  then one might perhaps expect to find $\pi(gf) = \pi\big( \pi(g)\eta(f)\big)\, 
  \pi(f)$ and $\eta(gf) = \eta(g) \, 
  \eta\big(\pi(g)\eta(f)\big)$. While the second equality does 
  hold, the first does {\em not} hold in general, not even in the category 
  $\Fin$. The reason is the local nature of the condition `fibrewise 
  order-preserving': just because it holds over $S$ does not mean it 
  holds over $R$. In fact one easily gets counterexamples with $R=1$.
\end{warning}

\begin{lemma}\label{strongff}
  Let $\catO$ be a \ff operadic category. Then the following statements
  are equivalent:
  \begin{enumerate}
	\item $\catO$ is strongly factorisable.
	\item Any order-preserving quasibijection in $\catO$ is the identity.
  \end{enumerate}
  In this case the counit of the adjunction $\rrr\dashv \iii$
  is necessarily the identity.
\end{lemma} 

\begin{proof}
  The implication $(1)\dbto(2)$ is the content of Lemma~\ref{sf in Sydney} 

  For the converse implication $(2)\dbto(1)$, we need to prove that a
  \fop pita factorisation as provided by the adjunction $\rrr\dashv \iii$
  is unique. So consider two different \fop factorisations of $f:T\to S$,
  \begin{equation*}
    \xymatrix@C = +3em{
     T \ar[d]_{\pi(f)} \ar[r]^{\sigma}& T'' \ar[d]^{\lambda}
      \\
    T' \ar[r]_{\eta(f)}\ar@{-->}[ru]^{\alpha} & S
  }  
  \end{equation*}
  where the left bottom corner comes from the adjunction. Then, by the
  universal property of the unit, there is a unique lifting $\alpha$.

  Notice that we have a commutative triangle for each $i\in |S|$: 
   \[
    \xymatrix@C = +1em@R = +1em{
  f^{-1}(i)     \ar[rr]^{\sigma_i} \ar[dr]_{\pi(f)_i} & & \lambda^{-1}(i) 
      \\
      &\eta(f)^{-1}(i)\ar[ur]_{\alpha_i}&
    }
\]
  In this diagram, $\sigma_i$ and $\pi(f)_i$ are identities, because they
  are order-preserving quasibijections. Hence $\alpha_i$ is the identity,
  and moreover, $\alpha$ is a quasibijection, because the fibres of
  $\alpha$ coincide with the fibres of $\alpha_i, i\in |S|$ by
  Axiom~\ref{axQ:BM-fibres-of-local-fibres}. Observe then that $|\alpha|$
  is an identity since $|\lambda|$ and $|\eta(f)|$ are order-preserving
  morphisms in $\Fin$. So $\alpha$ is an order-preserving quasibijection
  in $\catO$, and is therefore an identity by assumption.

  The counit of the adjunction $\rrr\dashv \iii$ is invertible because
  $\iii$ is fully faithful. So the component of this counit at an
  order-preserving morphism $f$ must be the inverse to $\pi(f)$. But $f$
  admits a canonical pita factorisation as the identity followed by $f$.
  Since $\catO$ is strongly factorisable, we have $\pi(f)=\id$.
\end{proof} 

\begin{definition}
  A \ff operadic category satisfying the conditions of Lemma~\ref{strongff}
  is called {\em strongly \ff}.
\end{definition}

In a strongly \ff operadic category $\catO$, we now consider three functions
on the components of the nerve of $\catO$:
\begin{enumerate}
  \item $\pi(-):\catO_1\to \catO_1$ which maps $f$ to the component of the
  unit $\pi(f)$ of the adjunction $\rrr\dashv\iii$.
  
  \item $\eta(-):\catO_1\to\catO_1$ where $\eta(f) = \rrr(f)$.
  
  \item $\eta(-/-):\catO_2\to \catO_1$, which for a composable pair $T
  \stackrel{f}\to S \stackrel{g}\to R$ (as in \eqref{chainfg}) associates
  the morphism $\eta(f/g)$ as in Proposition~\ref{hitrayadiagramma}.
\end{enumerate} 
The following lemma summarises relations enjoyed by these functions, which
will be very useful later.
 
\begin{lemma}\label{etafunction}
  In a strongly \ff operadic category $\catO$, with reference to 
  $T \stackrel{f}\to S \stackrel{g}\to R \stackrel h \to Q$ one has:
  \begin{enumerate}
	\item $\pi(\pi(f))= \id$,  $\eta(\eta(f))= \eta(f)$, $\pi(\eta(f))= \id$, $\eta(\pi(f)) = \id;$
	\item $\eta(g)\eta(f/g) = \eta(gf);$
	\item $\eta(g/h)\eta(f/hg) = \eta(gf/h);$ 
	\item $\eta({f}/{\id}) = \eta(f);$ 
	\item $\eta(\id/f) = \id$;
	\item If $g$ and $gf$ are order-preserving, then $\eta(f/g) = f$.
  \end{enumerate}
\end{lemma} 

\begin{lemma}\label{pimonotone}
  In a strongly \ff operadic category $\catO$, the naturality squares
  \begin{equation}\label{upperfop}
	  \xymatrix@C = +3em{
	   T \ar[d]_{f} \ar[r]^{\pi(gf)}& T' \ar[d]^{\eta(f/g)}
		\\
	  S \ar[r]_{\pi(g)} & S'
  }    
  \end{equation}
  of the unit $\pi$ of $\rrr\dashv\iii$ are
  \fop squares.
\end{lemma}
 
\begin{proof} 
  Consider the commutative diagram
  \begin{equation*}
    \xymatrix@C = +4em{
     T \ar[d]_{f} \ar[r]^{\pi(gf)}
     & T' \ar[d]_{\eta(f/g)} \ar[r]^{\id}
     & T' \ar[d]^{\eta(gf)} 
      \\
    S \ar[r]_{\pi(g)} & S' \ar[r]_{\eta(g)} & R .
  }
  \end{equation*}
  The right-hand square is obviously \fop, and the composite 
  rectangle is \fop since it is precisely the pita factorisation of $gf$.
  Proposition~\ref{3out2} now ensures that also the left-hand 
  square is \fop, as asserted.
\end{proof}

\section{Strictly factorisable operadic categories}

Being a strongly \ff operadic category is still an extra structure on $\catO$, not
just a property. But the only piece of this structure is a choice of the morphism
$\eta(f/g)$ for each composable pair $T \stackrel f \to S \stackrel g \to R$. (This
choice is the choice of a left adjoint $\rrr$.)
The situation where this choice is unique is important enough to deserve a definition:
\begin{definition}\label{strict-ff}
  A strongly factorisable operadic category $\catO$ is called {\em 
  \strictlyfactorisable}
  if for each chain of morphisms $T\stackrel{f}{\to} S\stackrel{g}{\to} R$ there
  exists a unique morphism $\eta(f/g)$ making the diagram
  \eqref{etafunctorial1} commutative.
\end{definition}

\begin{lemma}\label{strongtostrict}
  A \strictlyfactorisable operadic category is also functorially factorisable.
\end{lemma}

\begin{proof}
  We only need to establish the functoriality relations (1) and (2) from
  Proposition~\ref{hitrayadiagramma}. But they follow readily from the
  uniqueness of $\eta(f/g)$.
\end{proof}

Another obvious characterisation of strictly \ff operadic categories is 
given in the following
\begin{lemma}
  A strongly factorisable operadic category is \strictlyfactorisable if and only if
  for any $R\in \catO$ there exists a unique left adjoint $\rrr$
  satisfying Definition~\ref{funfactor}.
\end{lemma}

\begin{lemma}
  The category $\Fin$ is \strictlyfactorisable.
\end{lemma}
\begin{proof}
  This is a direct check. See also Example~\ref{ex:WBU}
\end{proof}

More general examples of \strictlyfactorisable operadic categories can be established 
using blow-up axioms.
Recall \cite[p.9--10]{Batanin-Markl:1812.02935}, \cite[p.4--5]{Batanin-Markl:2105.05198}
that $\catO$ is said to satisfy the {\em weak blow-up axiom} if
for each order-preserving morphism $h:T\to R$ and each $\norm{R}$-indexed 
family of morphisms
$f_i:h^{-1}(i)\to F_i, i\in |R|$ there exists a unique factorisation
 \begin{equation*} 
    \xymatrix@C = +1em@R = +1em{
      T      \ar[rr]^f \ar[dr]_h & & S \ar[dl]^g
      \\
      &R&
    }
\end{equation*} 
such that $g$ is an order-preserving morphism and the induced morphisms on
fibres are exactly $f_i:h^{-1}(i)\to F_i, i\in |R|$.

\begin{lemma}\label{uniqueeta}
  Let $\catO$ be a strongly \ff operadic category in which the weak blow-up axiom is
  satisfied. Then $\catO$ is \strictlyfactorisable.
 \end{lemma}  

\begin{proof} 
  Since $\catO$ is strongly \ff, there exists at least one filler $\eta(f/g)$
  in the diagram \eqref{etafunctorial1}. Let $\alpha$ be yet another
  filler. Then from the commutativity of the top square we have
  $$
  \alpha_i = \alpha_i\pi(gf)_i = \pi(g)_i f_i = f_i \ \text{for all} \  i\in |R|.
  $$ 
  Thus $\eta(f/g)$ and $\alpha$ induce equal morphisms on fibres over $R$
  and both $\eta(gf)$ and $\eta(g)$ are order-preserving. Hence, by weak
  blow-up we conclude that $\eta(f/g) = \alpha$.
\end{proof} 

\begin{proposition}\label{qbwbff}
  In a factorisable category $\catO$, if all quasibijections are invertible
  and the weak blow-up axiom is satisfied, then $\catO$ is a 
  \strictlyfactorisable
  operadic category.
\end{proposition} 

\begin{proof}
  It was proved in \cite[Lemma~2.11]{Batanin-Markl:1812.02935}
  that the conditions imply that the operadic
  category $\catO$ is strongly factorisable. To prove that it is \ff we
  define $\eta(f/g) = \pi(g)f\pi(gf)^{-1}$. By Lemma~\ref{uniqueeta},
  $\catO$ is \strictlyfactorisable.
\end{proof} 

\begin{example}\label{ex:WBU}
  The category $\Fin$ satisfies the weak blow-up axiom 
  (see \cite{Batanin-Markl:2105.05198}) and has invertible 
  quasibijections, so it is \strictlyfactorisable. Similarly, the 
  operadic
  category $\Fin_{\mathrm{surj}}$ is \strictlyfactorisable. A more 
  complicated example is
  the operadic category of graphs: it is shown in \cite[3.15--3.18]{Batanin-Markl:2105.05198}
  that it satisfies the conditions of
  Proposition~\ref{qbwbff}.  For many other examples, see
  \cite{Batanin-Markl:1812.02935,Batanin-Markl:2105.05198}.
\end{example}

The following feature, shared with orthogonal factorisation systems, is a
consequence of functoriality, but note that due to the local
nature of pita factorisation, the map $\tau$ on domains must be required
to be a quasibijection:

\begin{proposition}\label{fundamental}
  Let $\catO$ be a \strictlyfactorisable operadic category
  and consider a commutative square
\begin{equation}\label{basic}
    \xymatrix@C = +3em{
     T \ar[d]_{f} \ar[r]^{\sigma}& T' \ar[d]^{g}
      \\
    S \ar[r]_{\tau} & S'
}    
\end{equation}
in which  
$\tau$ is a quasibijection. Then:
\begin{enumerate} 
\item
There exists a unique morphism $\omega=\omega(\sigma,\tau) : T_1 \to
T'_1$ making the following diagram commutative
 \begin{equation}\label{prism}
    \xymatrix@C = +4em{
     T\ar@/^-3.3ex/[dd]_{f}
      \ar[d]^{\pi(f)} \ar[r]^{\sigma}& T' \ar[d]_{\pi(g)}
      \ar@/^3.3ex/[dd]^{g} 
           \\
    T_1 \ar[d]^{\eta(f)}\ar[r]^{\omega(\sigma,\tau)} & T'_1 \ar[d]_{\eta(g)}  
      \\ 
     S \ar[r]_{\tau} & S'    
}    
\end{equation}
\item If $\sigma$ is a quasibijection which is
\fop with respect to $\tau$, then $\omega$ is also \fop with respect to $\tau$ and
 \begin{enumerate}
 \item $\omega=\pi(\tau\eta(f))$,
  \item $\eta(g) = \eta(\tau\eta(f))$.  
\end{enumerate}
\end{enumerate} 
\end{proposition}

\begin{proof}
  To prove the first statement, consider the following factorisations in the
  square \eqref{basic}:

\begin{equation}\label{sqfactorisation}
    \xymatrix@R=+1em@C = +3em{
     T \ar[ddd]_{f} \ar[dr]^{\pi(f)}\ar[rrr]^{\sigma}&  & &T'\ar[dl]^{\pi(g)} \ar[ddd]^{g}
     \\
      &T_1\ar[ddl]^{\eta(f)}\ar[dr]_{\pi(\tau\eta(f))}\ar@{-->}[r]^{\omega} & T'_1\ar[ddr]^{\eta(g)}&
     \\
     & &T_2\ar[dr]_(0.3){\eta(\tau\eta(f))} \ar@{.>}[u]&
      \\
    S \ar[rrr]_{\tau} & & & S'
}    
\end{equation}
The part of this diagram under the diagonal can be written as a pasting of squares:
 \begin{equation}\label{past}
    \xymatrix@C = +4em{
     T \ar[d]_{f} \ar[r]^{\pi(f)}
     & T_1 \ar[d]_{\eta(f)} \ar[r]^{\pi(\tau\eta(f))}
     & T_2 \ar[d]^{\eta(\tau\eta(f))}
      \\
    S \ar[r]_{\id} & S \ar[r]_{\tau} & S' \,.
}    
\end{equation}
Notice that both squares are \fop, hence their pasting is also \fop by 
Lemma~\ref{3out2}~(3).
Moreover, the composite $\pi(\tau\eta(f))\pi(f)$ is a quasibijection and
$\eta(\tau\eta(f))$ is order-preserving. It follows that this pasting represents the
canonical pita factorisation of $\tau f$. By strict functorial
factorisability we have an arrow $$\eta(\sigma/g): T_2\to T_1$$ fitting in
the commutative diagram \eqref{sqfactorisation} as the vertical dashed
arrow. Now we define
$$
\omega = \eta(\sigma/g)\pi(\tau\eta(f))
$$
shown in the diagram as a horizontal dashed arrow.
We need to show that $\omega$ is unique. 
Assume $\omega':T_1\to T'_1$ is another filler. Then we have a 
commutative square
\begin{equation*}
    \xymatrix@C = +4em{
     T_1 \ar[d]_{\pi(\tau\eta(f))} \ar[r]^{\omega'}& T'_1 \ar[d]^{\eta(g)}
      \\
    T_2 \ar[r]_{\eta(\tau\eta(f))}\ar@{-->}[ru]^{\alpha} & S
}    
\end{equation*}
 with $\eta(g)$ order-preserving. The property of the unit of the
 adjunction now gives a unique lifting $\alpha$ in
 this square and $\omega'$, and in particular a factorisation $\omega'=
 \alpha\pi(\tau(\eta(f))$. Observe also that $\alpha$ makes the following
 diagram commutative:
 \begin{equation}\label{etafunctorial}
    \xymatrix@C = +1.5em@R = +1em{
    T \ar@/^-4.3ex/[ddrrr]_{\tau f} \ar[drr]^(0.6){\hskip0.3em\pi(f)\pi(\tau\eta(f))}     \ar[rrrrrr]^{\sigma} & & & & & & T' \ar@/^4.3ex/[ddlll]^{g} \ar[dll]_(0.6){\hskip-0.7em\pi(g)} 
     \\
     & &T_2      \ar[rr]^{\alpha} \ar[dr]_(0.2){\eta(\tau\eta(f))\hskip-0.5em} & & T'_1 \ar[dl]^(0.2){\eta(g)} & &
      \\ 
     & & &S'& & &
    }
\end{equation}
The uniqueness property of \strictlyfactorisable operadic categories (\ref{strict-ff}) now 
tell us
$\alpha = \eta(\sigma/g)$, and therefore $\omega' = \omega$. 

For the second statement, consider a commutative diagram over $S'$: 
\begin{equation*}
    \xymatrix@C = +4em{
T \ar[d]_{\pi(f)} \ar[r]^{\sigma}& T' \ar[d]^{\pi(g)} \ar@/^3.3ex/[ddr]^g & 
      \\
  T_1 \ar[r]^{\omega} \ar@/^-3.3ex/[drr]_{\tau\eta(f)}
  & T'_1\ar[dr]^{\eta(g)} &
  \\ 
  & & S'   .
}    
\end{equation*}
For each $i\in |S'|$ we have the induced commutative square 
\begin{equation*}
    \xymatrix@C = +3em{
     (\sigma g)^{-1}(i) \ar[d]_{\pi(f)_i} \ar[r]^{\sigma_i}& g^{-1}(i) \ar[d]^{\pi(g)_i}
      \\
   (\tau\eta(f))^{-1}(i) \ar[r]_{\omega_i} & \eta(g)^{-1}(i) .
}    
\end{equation*}
  The pasting \eqref{past} shows that $\pi(f)_i$ is an order-preserving
  quasibijection; it must therefore be the identity. Similarly, $\pi(g)_i$ is an
  identity. Finally, if $\sigma$ is a quasibijection which is 
  \fop with respect to
  $\tau$, then $\sigma_i$ is the identity too, and therefore $\omega_i$ is
  the identity. By Axiom~\ref{axQ:BM-fibres-of-tau-maps},
  the fibres of $\omega$ are
  equal to the fibres of $\omega_{|g|(j)}$, and so the fibres of $\omega$ are
  trivial. Hence $\omega$ is a quasibijection.

  Since $\omega$ provides a factorisation of $\tau\eta(f)$ into a \fop
  quasibijection followed by an order-preserving morphism $\eta(g)$, it must
  be equal to $\pi(\eta(f)\tau)$ and therefore we have $\eta(g) = \eta(\eta(f)\tau)$.
\end{proof} 

\section{Locally order-preserving chains}

For the remainder of the paper, we fix a \strictlyfactorisable operadic
category $\catO$. In this section we study categories of certain locally
order-preserving chains of fixed length $n$. In the next section, we let
$n$ vary and assemble these categories into a toplax simplicial category,
called the pita nerve (it's a very special kind of a oplax simplicial
category, almost strict).

\begin{definition}\label{def:fopdiagram}
  A {\em \fop diagram} in $\catO$ is a commutative diagram
  \begin{equation}\label{Wmorphisms}
  \begin{tikzcd}[column sep={2em,between origins}]
     T_{n} \ar[d, "f_{n}"'] \ar[rr, "\sigma_n"] && S_{n} \ar[d, "g_n"] 
	 \\
     \vdots \ar[d, "f_1"'] & \vdots & \vdots \ar[d, "g_1"]
      \\ 
     T_0 \ar[rr, "\sigma_0"'] && S_0   
  \end{tikzcd}
  \end{equation}
  in which  $\sigma_i$ is \fop with respect to $\sigma_0$ for all $ 0 \le i \le n$. 
\end{definition}

We first define a category $\wnerve(\catO)_n$ (for each $n\geq 0$). The
objects are chains of morphisms
$$
T_n\stackrel{f_n}{\to} \ldots \stackrel{f_1}{\to} T_0.
$$
A morphism between two chains is a \fop diagram \eqref{Wmorphisms} in $\catO$ (in the sense of 
Definition~\ref{def:fopdiagram})
in which all horizontal arrows are quasibijections. It follows from
Proposition~\ref{verticalfop} that each commutative square in this diagram
is a \fop square. Proposition~\ref{3out2}~(3) allows us to compose these
morphisms (horizontally).

\begin{definition}
  A chain of morphisms in $\catO$
  $$
  T_n\xto{f_n} \ldots \xto{f_1} T_0
  $$
  is called a {\em locally order-preserving chain} if all composites
  $$
  T_k \stackrel{f_k}\to \cdots \stackrel{f_1}\to T_0, \quad k=0,\ldots,n
  $$
  ending in $T_0$ are order-preserving morphisms. The word `locally' thus
  means `locally at the last object in the chain'. Note that this does
  not imply that other morphisms in the chain are necessarily
  order-preserving. Even in the category $\Fin$ that is not generally the
  case.
\end{definition} 
 
Let 
$$
\iii_n: \fnerve(\catO)_n \subset \wnerve(\catO)_n
$$ 
be the full subcategory of $\wnerve(\catO)_n$ consisting of locally order-preserving chains.

\begin{proposition}\label{wfn}
  For each $n$, the subcategory $\fnerve(\catO)_n$ is a reflective subcategory with a
  unique reflection
  $$
  \rrr_n:\wnerve(\catO)_n\to \fnerve(\catO)_n.
  $$
\end{proposition} 

\begin{proof} 
 For $n=0$ we have $\wnerve(\catO)_0 = \fnerve(\catO)_0$.

Let $n=1$. We define the reflection $\rrr_1$ on an object $f:T_1\to T_0$
of $\wnerve(\catO)_1 $ by pita factorisation: $\rrr_1(f)= \eta(f):
T_1'\to T_0$. The unit of the adjunction $\unit : \id\to \iii_1\rrr_1$ is
given by the $\pi$-factor of the pita factorisation:
\begin{equation*}
    \xymatrix@C = +3em{
     T_1 \ar[d]_{f} \ar[r]^{\pi(f)}& T_1' \ar[d]^{\eta(f)}
      \\
    T_0 \ar[r]_{\id} & T_0 ,
}    
\end{equation*}
clearly a \fop square.
The counit of the adjunction $\rrr_1\dashv \iii_1$ is the identity 
(by~\ref{strongff}). 

 The value of the reflection  $\rrr_1$ on a morphism 
  \begin{equation}\label{wmorph}
    \xymatrix@C = +2em{
     T_{1} \ar[d]_{f} \ar[r]^{\sigma}& S_{1} \ar[d]^{g} 
 \\
     T_0 \ar[r]_{\tau} & S_0   
}    
\end{equation}   
  is defined by the commutative square from Proposition~\ref{fundamental}:
 \begin{equation*}
    \xymatrix@C = +7em{
     T'_{1} \ar[d]_{\eta(f)} \ar[r]^{\omega(\sigma,\tau)= \pi(\tau\eta(f))}& S'_{1} \ar[d]^{\eta(g)} 
 \\
     T_0 \ar[r]_{\tau} & S_0   
}    
\end{equation*}  
The functoriality of $\rrr$ and naturality of the unit is again the content
of Proposition~\ref{fundamental}.

We consider now the case $n=2$. Let
\begin{equation}\label{TfSgR}
  T\xto{f}S\xto{g} R
\end{equation}
be an object of $\wnerve(\catO)_2$.  We then have a commutative diagram
 \begin{equation}
    \xymatrix@C = +5em{
     T\ar@/^-4.3ex/[dd]_{gf}
      \ar[d]^{f}
       \ar[r]^{\pi(gf)}& 
      T' \ar[d]_{\eta(f/g)}
      \ar@/^4.3ex/[dd]^{\eta(gf)} 
           \\
    S \ar[d]^{g}\ar[r]^{\pi(g)} & S' \ar[d]_{\eta(g)}
      \\ 
     R \ar[r]_{\id} & R    
}    
\end{equation}
In this diagram the right-hand column $T'\xto{\eta(f/g)} S' \xto{\eta(g)}
R$ is a locally order-preserving chain, which we take as value of
$\mathsf{r_2}$ on \eqref{TfSgR}. Furthermore, the two squares of the
diagram together constitute a morphism in $\wnerve(\catO)_2$, which we
take as the unit $\unit_2$ of the adjunction $\rrr_2\dashv \iii_2$. The
counit of this adjunction is the identity, as usual (by~\ref{strongff}).
 
To prove functoriality of $\mathsf{r}_2$, let  
 \begin{equation}\label{verticalcomposition}
    \xymatrix@C = +3em{
     T
      \ar[d]_{f} \ar[r]^{\omega}& T' \ar[d]^{a}
           \\
    S \ar[d]_{g}\ar[r]^{\sigma} & S' \ar[d]^{b}  
      \\ 
     R \ar[r]_{\rho} & R'    
}    
\end{equation}
be a commutative diagram in $\catO$, in which all horizontal maps are quasibijections
and both $\sigma$ and $\omega$ are \fop with respect to $\rho$. Then we define the 
value of $\mathsf{r}_2$ on \eqref{verticalcomposition} to be the diagram
 \begin{equation*}
    \xymatrix@C = +4em{
     T_1
      \ar[d]_{\eta(f/g)} \ar[r]^{\pi(\rho\eta(gf))}& T'_1 \ar[d]^{\eta(a/b)}
           \\
    S_1 \ar[d]_{\eta(g)}\ar[r]^{\pi(\rho\eta(g))} & S'_1 \ar[d]^{\eta(b)}  
      \\ 
     R \ar[r]_{\rho} & R'    \,,
}    
\end{equation*}
but we need to check various things before we can make sense of that.
Note first of all that the vertical composites match the value of $\mathsf{r}_2$ on
objects already defined. Second, we claim that both the horizontal arrows in the diagram
are quasibijections that are \fop with respect to $\rho$.
To see this, apply Proposition~\ref{fundamental} to the bottom square of
Diagram~\eqref{verticalcomposition} and to the whole composite square, to get the 
solid part of the
commutative diagram
  \begin{equation}\label{proizvodnyj}
    \xymatrix@C = +5em{
     T_1\ar@/^-4.3ex/[dd]_{\eta(gf)}
      \ar@{-->}[d]^{\eta(f/g)}
       \ar[r]^{\pi(\rho\eta(gf))}& 
      T'_1 \ar@{-->}[d]_{\eta(a/b)}
      \ar@/^4.3ex/[dd]^{\eta(ba)} 
           \\
    S_1 \ar[d]^{\eta(g)}\ar[r]_{\pi(\rho\eta(g))} & S'_1 \ar[d]_{\eta(b)}  
      \\ 
     R \ar[r]_{\rho} & R'    
}    
\end{equation}
Then we have $\eta(g)\eta(f/g) = \eta(gf)$ and $\eta(b) = \eta(a/b)\eta(b)$ by 
Lemma~\ref{etafunction}, giving the commutativity of the triangles in the diagram.
It remains to establish that the top square commutes, which is the equality
$\eta(a/b)\pi(\rho\eta(gf)) =   \eta(f/g)\pi(\rho\eta(g))$.
This square appears as the middle horizontal square in the big diagram
\[
\begin{tikzcd}[column sep={48pt,between origins}, row sep={36pt,between origins}]
  & T \ar[ddd, pos=0.3, "\pi(gf)"] \ar[ld, "f" {anchor=south, rotate=35}] \ar[rrd, "\omega" {anchor=south, rotate=-19}] & &
  \\
  S \ar[ddd, "\pi(g)"'] \ar[rrd, pos=0.8, "\sigma" {anchor=north, rotate=-19}] & & & 
  T' \ar[ddd, "\pi(ba)"]  \ar[ld, "a" {anchor=north, rotate=35}]
  \\
  & & S' \ar[ddd, pos=0.17,"\pi(b)"'] &
  \\[-12pt]
  & T_1 \ar[ddd, pos=0.8, "\eta(gf)"'] \ar[ld, pos=0.45, "\eta(f/g)" {anchor=north, rotate=35}] \ar[rrd, 
  pos=0.75,
  "\pi(\rho\eta(gf))"  {anchor=south, rotate=-19}] & &
  \\
  S_1 \ar[ddd, "\eta(g)"'] \ar[rrd, pos=0.75, "\pi(\rho\eta(g))" {anchor=south, rotate=-19}] & & & T_1' \ar[ddd, 
  "\eta(ba)"] \ar[ld, pos=0.45, "\eta(a/b)" {anchor=north, rotate=35}]
  \\
  & & S_1'\ar[ddd, pos=0.7, "\eta(b)"'] &
  \\[-12pt]
  & R \ar[ld, equal] \ar[rrd, pos=0.8, "\rho" {anchor=south, rotate=-19}] & &
  \\
  R \ar[rrd, "\rho" {anchor=north, rotate=-19}] & & & R' \ar[ld, equal]
  \\
  & & R'
  \\
\end{tikzcd}
\]
Here the four vertical composites are pita factorisations. The
left-hand and right-hand composite faces commute because they are instances of
Diagram~\eqref{etafunctorial1}. The front and back composite faces commute
because they are instances of Proposition~\ref{fundamental} (2) (b). The top
face commutes by assumption, and the bottom face trivially commutes. The middle
horizontal square is now forced to commute by Proposition~\ref{fundamental} (1)
(uniqueness) applied to
\[
\begin{tikzcd}
T \ar[d, "gf"']\ar[r, "a\omega=\sigma f"] & S' \ar[d, "b"]  \\
R \ar[r, "\rho"'] & R' .
\end{tikzcd}
\]

For an arbitrary $n\ge 0$, we define the value of $\mathsf{r_n}$ on a chain
$T_n\stackrel{f_n}{\to} \cdots \stackrel{f_1}{\to} T_0$ to be equal to
$$
T_n'\xto{\eta(f_n/f_1\cdots f_{n-1})} T'_{n-1}\xto{\eta(f_{n-1}/f_1\cdots
f_{n-2})} \quad \cdots \quad
\xto{\eta(f_1)} T_0.
$$
The unit of the adjunction is given by the \fop diagram
\begin{equation}
    \xymatrix@C = +2em{
     T_{n} \ar[d]_{f_{n}} \ar[rr]^{\pi(f_1\cdots f_{n-1}f_n)}& &T'_{n} 
     \ar[d]^{\eta(f_n/f_1\cdots f_{n-1})} 
           \\
     \vdots \ar[d]_{f_3} & \vdots & \vdots \ar[d]^{\eta(f_3/f_1f_2)}\\  
       T_2\ar[d]_{f_2} \ar[rr]^{\pi(f_1f_2)}& & T_2' \ar[d]^{\eta(f_2/f_1)}\\
 T_1  \ar[d]_{f_1}\ar[rr]^{\pi(f_1)}  & &T_1'\ar[d]^{\eta(f_1)}
      \\ 
     T_0 \ar[rr]_{\id} & &T_0 .   
}    
\end{equation}   
For functoriality of $\rrr_n$ we observe that 
a commutative diagram in $\catO$
\begin{equation}
    \xymatrix@C = +1em{
     T_{n} \ar[d]_{f_{n}} \ar[rr]^{\sigma_n}&& S_{n} \ar[d]^{g_n} 
           \\
   \ar[d]_{f_1} \vdots& \vdots& \vdots\ar[d]^{g_1}
      \\ 
     T_0 \ar[rr]_{\sigma_0} && S_0   
}    
\end{equation}   
 in which all horizontal maps are quasibijections and $\sigma_i$ are \fop with respect to $\sigma_0$ 
 induces a commutative diagram
 \begin{equation}\label{WTmorphisms}
  \begin{tikzcd}[column sep={3.5em,between origins}]
     T'_{n} \ar[d, "\eta(f_{n}/f_{1}\cdots f_{n-1})"']
     \ar[rr, "\pi(\sigma_0\eta(f_1\cdots f_n))"]
     & &
     S'_{n} \ar[d, "\eta(g_{n}/g_{1}\cdots g_{n-1})"]
           \\
   \vdots \ar[d, "\eta(f_1)"'] & \vdots
   & 
   \vdots\ar[d, "\eta(g_1)"]
      \\ 
     T_0 \ar[rr, "\sigma_0"'] & & S_0   
	 \end{tikzcd}
\end{equation}   
by an induction similar to the case $n=2$.
\end{proof}

\begin{definition}
  We define $(\tttt_n,\id,\unit_n)$ to be the idempotent monad on $\wnerve(\catO)_n$
  induced by the adjunction $\rrr_n\dashv \iii_n$. 
\end{definition} 

\begin{lemma}\label{df}
  The functor $\fnerve(\catO)_n\to \fnerve(\catO)_0$ that sends a
  locally order-preserving chain $T_n\xto{f_n} \cdots \xto{f_1} T_0$ to $T_0$ and a
  morphism \eqref{Wmorphisms} to $\sigma_0$ is a discrete opfibration.
\end{lemma}
 
\begin{proof}
  Indeed, if a chain $T_n\stackrel{f_n}{\to} \ldots \stackrel{f_1}{\to}
  T_0$ is locally order-preserving, then any $\sigma_0:T_0\to S_0$ admits a unique
  lifting to a morphism in $\fnerve(\catO)_n$ using  Diagram~\eqref{WTmorphisms}.
\end{proof}

\section{Toplax simplicial objects}

In the next section we shall assemble the categories $\fnerve(\catO)_n$
into some kind of simplicial object called the {\em pita nerve}. It is
not quite a simplicial object but rather what we call a {\em toplax
simplicial object}. The elementary definition of this notion may appear
somewhat specialised and ad hoc, but in this section we explain how it is
in fact a very natural notion, springing naturally from the fundamental
notion of decalage.

In elementary terms, a toplax simplicial object in a $2$-category 
$\KK$ should be a special case of normal oplax simplicial object (meaning an 
oplax functor $\simplexcategory\op\to\KK$) whose non-strictness is
concentrated in the top face operators in the following way. There are
two conditions:
\begin{enumerate}
  \item The restriction to a $\Dt$-presheaf is strict.

  \item For every triangle given by two composable generators of
  $\simplexcategory$ that are {\em not} two top face operators, the 
  corresponding $2$-cell in $\KK$ is strict.
\end{enumerate}

The same structure can be described in term of squares instead of 
triangles (cf.~Remark~\ref{Jardine-rmk} further below). This viewpoint
takes as starting point the simplicial identities, but replaces one 
family of identities with non-invertible $2$-cells subject to some 
equations:

\begin{definition}\label{def:toplax}
  A {\em toplax simplicial object} in a $2$-category $\KK$ consists of
  
  \begin{enumerate}
    \item   a sequence of object $X_n$, $n\geq 0$;

    \item face and degeneracy operators $d_i$, $s_i$ like in a 
    simplicial object;
  
    \item $2$-cells (for $n\geq 0$)
\[
\begin{tikzcd}
X_{n+2} \ar[d, "d_{n+1}"'] \ar[r, "d_{n+2}"] & X_{n+1} \ar[d, "d_{n+1}"]  \\
X_{n+1} \ar[r, "d_{n+1}"'] \ar[ru, Rightarrow, shorten <=14pt, shorten 
>=14pt, "\beta_n"]& X_n \,.
\end{tikzcd}
\]

  \end{enumerate}
  
  These data are subject to the following axioms.
  
  \begin{enumerate}
    \item All the strict simplicial identities except the ones replaced 
    by $\beta_n$ (that is, all simplicial identities except those involving 
    two consecutive top face operators);
  
    \item The following four series of equations for the 
    $2$-cells $\beta_n$:
  \end{enumerate}

\begin{equation}\label{beta-eq}
\begin{tikzcd}[column sep={44pt,between origins},row sep={35pt,between origins}]
X_{n+3} \ar[dd, "d_{n+1}"'] \ar[rd, pos=0.7, "d_{n+2}"'] 
\ar[rr, "d_{n+3}"] && X_{n+2}  \ar[rd, "d_{n+2}"] &  
\\
& X_{n+2} \ar[dd, "d_{n+1}"] \ar[rr, "d_{n+2}"'] 
\ar[ru, Rightarrow, shorten <=8pt, shorten >=8pt, "\beta_{n+1}"]
&& X_{n+1} \ar[dd, "d_{n+1}"] 
 \\
 X_{n+2} \ar[rd, "d_{n+1}"'] \ar[ru, phantom, "\commutes" description] & & &
 \\
 & X_{n+1} \ar[rr, "d_{n+1}"'] 
 \ar[rruu, Rightarrow, shorten <=36pt, shorten >=36pt, "\beta_n"']&& X_n
\end{tikzcd}
\ \ \
= 
\ \ \
\begin{tikzcd}[column sep={44pt,between origins},row sep={35pt,between origins}]
X_{n+3} \ar[dd, "d_{n+1}"']  
\ar[rr, "d_{n+3}"] && X_{n+2}  \ar[dd, "d_{n+1}"']\ar[rd, "d_{n+2}"] & 
\\
 & && X_{n+1} \ar[dd, "d_{n+1}"]
 \\
 X_{n+2} \ar[rd, "d_{n+1}"'] \ar[rr, "d_{n+2}"] 
 \ar[rruu, phantom, "\commutes" description] & & X_{n+1} \ar[rd, 
 "d_{n+1}"'] 
 \ar[ru, Rightarrow, shorten <=8pt, shorten >=8pt, "\beta_n"] &
 \\
 & X_{n+1} \ar[rr, "d_{n+1}"'] 
 \ar[ru, Rightarrow, shorten <=8pt, shorten >=8pt, "\beta_n"]&& X_n
\end{tikzcd}
\end{equation}
\begin{equation}\label{beta-eq-bis}
\begin{tikzcd}[column sep={44pt,between origins},row sep={35pt,between origins}]
X_{n+2}  \ar[rd, pos=0.7, "d_{n+1}"'] 
\ar[rr, "d_{n+2}"] && X_{n+1}  \ar[rd, "d_{n+1}"] &  
\\
& X_{n+1} \ar[rr, "d_{n+1}"'] 
\ar[ru, Rightarrow, shorten <=8pt, shorten >=8pt, "\beta_{n}"]
&& X_{n} 
 \\
 X_{n+1} \ar[uu, "s_{n+1}"]  \ar[rd, "\id"'] \ar[ru, phantom, "\commutes" description] & & &
 \\
 & X_{n+1} \ar[rr, "d_{n+1}"'] \ar[uu, "\id"']
 \ar[rruu, phantom, "\commutes" description]&& X_{n}  \ar[uu, "\id"']
\end{tikzcd}
\ \ \
= 
\ \ \
\begin{tikzcd}[column sep={44pt,between origins},row sep={35pt,between origins}]
X_{n+2}  
\ar[rr, "d_{n+2}"] && X_{n+1} \ar[rd, "d_{n+1}"] & 
\\
 & && X_{n} 
 \\
 X_{n+1} \ar[uu, "s_{n+1}"] \ar[rd, "\id"'] \ar[rr, "\id"] 
 \ar[rruu, phantom, "\commutes" description] & & X_{n+1} \ar[rd, 
 "d_{n+1}"'] \ar[uu, "\id"]
 \ar[ru, phantom, "\commutes" description] &
 \\
 & X_{n+1} \ar[rr, "d_{n+1}"'] 
 \ar[ru, phantom, "\commutes" description]&& X_n
 \ar[uu, "\id"']
\end{tikzcd}
\end{equation}

\begin{equation}\label{new-eq}
\begin{tikzcd}[column sep={44pt,between origins},row sep={35pt,between origins}]
X_{n+3} \ar[dd, "d_{i}"'] \ar[rd, pos=0.7, "d_{n+2}"'] 
\ar[rr, "d_{n+3}"] && X_{n+2}  \ar[rd, "d_{n+2}"] &  
\\
& X_{n+2} \ar[dd, "d_{i}"] \ar[rr, "d_{n+2}"'] 
\ar[ru, Rightarrow, shorten <=8pt, shorten >=8pt, "\beta_{n+1}"]
&& X_{n+1} \ar[dd, "d_{i}"] 
 \\
 X_{n+2} \ar[rd, "d_{n+1}"'] \ar[ru, phantom, "\commutes" description] & & &
 \\
 & X_{n+1} \ar[rr, "d_{n+1}"'] 
 \ar[rruu, phantom, "\commutes" description]&& X_n
\end{tikzcd}
\ \ \
= 
\ \ \
\begin{tikzcd}[column sep={44pt,between origins},row sep={35pt,between origins}]
X_{n+3} \ar[dd, "d_{i}"']  
\ar[rr, "d_{n+3}"] && X_{n+2}  \ar[dd, "d_{i}"']\ar[rd, "d_{n+2}"] & 
\\
 & && X_{n+1} \ar[dd, "d_{i}"]
 \\
 X_{n+2} \ar[rd, "d_{n+1}"'] \ar[rr, "d_{n+2}"] 
 \ar[rruu, phantom, "\commutes" description] & & X_{n+1} \ar[rd, 
 "d_{n+1}"'] 
 \ar[ru, phantom, "\commutes" description] &
 \\
 & X_{n+1} \ar[rr, "d_{n+1}"'] 
 \ar[ru, Rightarrow, shorten <=8pt, shorten >=8pt, "\beta_n"]&& X_n
\end{tikzcd}
\end{equation}
\begin{equation}\label{new-eq-bis}
\begin{tikzcd}[column sep={44pt,between origins},row sep={35pt,between origins}]
X_{n+3}  \ar[rd, pos=0.7, "d_{n+2}"'] 
\ar[rr, "d_{n+3}"] && X_{n+2}  \ar[rd, "d_{n+2}"] &  
\\
& X_{n+2} \ar[rr, "d_{n+2}"'] 
\ar[ru, Rightarrow, shorten <=8pt, shorten >=8pt, "\beta_{n+1}"]
&& X_{n+1} 
 \\
 X_{n+2} \ar[uu, "s_{i}"]  \ar[rd, "d_{n+1}"'] \ar[ru, phantom, "\commutes" description] & & &
 \\
 & X_{n+1} \ar[rr, "d_{n+1}"'] \ar[uu, "s_{i}"']
 \ar[rruu, phantom, "\commutes" description]&& X_{n}  \ar[uu, "s_{i}"']
\end{tikzcd}
\ \ \
= 
\ \ \
\begin{tikzcd}[column sep={44pt,between origins},row sep={35pt,between origins}]
X_{n+3}  
\ar[rr, "d_{n+3}"] && X_{n+2} \ar[rd, "d_{n+2}"] & 
\\
 & && X_{n+1} 
 \\
 X_{n+2} \ar[uu, "s_{i}"] \ar[rd, "d_{n+1}"'] \ar[rr, "d_{n+2}"] 
 \ar[rruu, phantom, "\commutes" description] & & X_{n+1} \ar[rd, 
 "d_{n+1}"'] \ar[uu, "s_{i}"]
 \ar[ru, phantom, "\commutes" description] &
 \\
 & X_{n+1} \ar[rr, "d_{n+1}"'] 
 \ar[ru, Rightarrow, shorten <=8pt, shorten >=8pt, "\beta_n"]&& X_n
 \ar[uu, "s_{i}"']
\end{tikzcd}
\end{equation}
The first two equations, \eqref{beta-eq} and \eqref{beta-eq-bis}, are the most
substantial ones.
Equations~\eqref{new-eq} are a variant of \eqref{beta-eq}, where the index $n+1$
of the vertical maps have been replaced by $i\leq n$. With this lesser index, more
squares in the cube become commutative, but there are still two beta cells left
to form an equation. Equations~\eqref{new-eq-bis} can also be viewed as variants
of \eqref{beta-eq-bis}, but with the degeneracy operators further away from the top,
so as not to behave like a section.
\end{definition}

It is Definition~\ref{def:toplax} that naturally arises from decalage, as we 
proceed to explain.
Recall that the decalage comonad is induced by the adjunction
\[
\begin{tikzcd}
\Dt \ar[d, shift left=2, "\uuu"]  \ar[d, phantom, "\scriptstyle\isleftadjointto" 
description]\\
\simplexcategory \ar[u, shift left=2, "\iii"]
\end{tikzcd}
\]
as
$$
\DDD := \iii\upperstar\uuu\upperstar  \,.
$$
Here $\Dt$ is the category of ordinals with a top element and 
top-preserving order-preserving maps, and $\iii\isleftadjointto\uuu$ is 
the free-forgetful adjunction.

It was observed in \cite{Garner-Kock-Weber:1812.01750} that 
$\DDD$-coalgebra structure, which in turn is equivalent to having extra 
top degeneracy maps so as to form a $\Dt$-presheaf, is equivalent to
having chosen local terminal objects. The same adjunction induces a
monad structure on $\Dt$-presheaves (that is, on the category of 
$\DDD$-coalgebras), which is denoted $\widetilde\DDD$ \cite[Definition 
9]{Garner-Kock-Weber:1812.01750}. The algebras for 
this monad are simplicial objects again! In detail, while coalgebras for
$\DDD$ amount to adding extra top degeneracy operators, which constitute
the coalgebra structure maps $s: X \to \DDD(X)$, in turn for $\DDD$-coalgebras
($\Dt$-presheaves) we already have those extra top degeneracy operators,
and $\widetilde\DDD$-algebra structure amounts to adding extra top face 
operators further on top -- these 
constitute the algebra structure map $d:\widetilde\DDD (X) \to X$ --  so as to obtain a 
simplicial object again (cf.~\cite[Lemma~11]{Garner-Kock-Weber:1812.01750}).

The only little twist we introduce in these fundamental relationships is
that we now consider presheaves with values in a $2$-category $\KK$, and
are concerned with {\em normal oplax} $\widetilde\DDD$-algebras instead
of strict algebras. We briefly recall these notions.

Let $\TTT : \mathcal{K} \to \mathcal{K}$ be a strict monad on a 
$2$-category $\mathcal{K}$. Recall that a normal oplax $\TTT$-algebra
is an triple $(A,a,\beta)$ consisting of an object $A\in \mathcal{K}$,
a $1$-morphism
$$
\TTT A \stackrel{a}\to A
$$
and a $2$-cell 
\begin{equation}\label{betaT}
\begin{tikzcd}
\TTT \TTT A \ar[r, "\TTT(a)"] \ar[d, "\mu_A"'] & \TTT A \ar[d, "a"]  \\
\TTT A \ar[r, "a"']  \ar[ru, Rightarrow, shorten <=12pt, shorten >=12pt, 
"\beta"] & A .
\end{tikzcd}
\end{equation}
This $2$-cell is subject to the two equations
\begin{equation}\label{TTTA}
\begin{tikzcd}[column sep={46pt,between origins},row sep={35pt,between origins}]
\TTT\TTT\TTT A \ar[dd, "\mu_{\TTT A}"'] \ar[rd, pos=0.7, "\TTT\mu_A"'] \ar[rr, 
"\TTT\TTT a"] && \TTT\TTT A  \ar[rd, "\TTT a"] &  
\\
 & \TTT\TTT A \ar[dd, "\mu_A"] \ar[rr, "\TTT a"'] \ar[ru, Rightarrow, 
 shorten <=10pt, shorten >=10pt, 
 "T\beta"]&& \TTT A \ar[dd, "a"] 
 \\
 \TTT\TTT A \ar[rd, "\mu_A"'] \ar[ru, phantom, "\commutes" description] & & &
 \\
 & \TTT A \ar[rr, "a"'] \ar[rruu, Rightarrow, shorten <=37pt, shorten 
 >=37pt, "\beta"]&& A
\end{tikzcd}
\quad
= 
\quad 
\begin{tikzcd}[column sep={46pt,between origins},row sep={35pt,between origins}]
\TTT\TTT\TTT A \ar[dd, "\mu_{\TTT A}"']  
\ar[rr, "\TTT\TTT a"] && \TTT\TTT A  \ar[dd, "\mu_A"']\ar[rd, "\TTT a"] & 
\\
 & && \TTT A \ar[dd, "a"]
 \\
 \TTT\TTT A \ar[rd, "\mu_A"'] \ar[rr, "\TTT a"] \ar[rruu, phantom, 
 "\commutes" description] & & \TTT A \ar[rd, "a"] \ar[ru, Rightarrow, 
 shorten <=10pt, shorten >=10pt, "\beta"] &
 \\
 & \TTT A \ar[rr, "a"'] \ar[ru, Rightarrow, shorten <=10pt, shorten >=10pt, "\beta"]&& A
\end{tikzcd}
\end{equation}
and
\begin{equation}\label{trapez}
\begin{tikzcd}
& \TTT\TTT A \ar[rd, "\TTT a"] \ar[dd, "\mu_A"]&   \\
\TTT A \ar[rd, "\id"'] \ar[ru, "\TTT \eta"]& \ar[l, phantom, pos=0.4, "\commutes" 
description]& \TTT A \ar[dd, "a"] \\
& \TTT A \ar[rd, "a"'] \ar[ru, Rightarrow, shorten <=14pt, shorten >=14pt, "\beta"]& \\
&& A
\end{tikzcd}
\qquad
=
\qquad
\begin{tikzcd}
& \TTT\TTT A \ar[rd, "\TTT a"] &   \\
\TTT A \ar[rd, "\id"'] \ar[ru, "\TTT \eta"] \ar[rr, "\id"']& 
\ar[u, phantom, pos=0.4, "\commutes" 
description]& \TTT A \ar[dd, "a"] \\
& \TTT A \ar[rd, "a"'] \ar[ru, phantom, pos=0.4, "\commutes" description]& \\
&& A
\end{tikzcd}
\end{equation}
In \eqref{TTTA}, the commutative square on the left is (strict) associativity of 
the monad multiplication, and the commutative square on the right is 
naturality. In \eqref{trapez},
the commutative triangle on the left is the (strict) unit law of the monad,
and the triangle on the right commutes because the algebra is required to be 
normal.

Instantiating now this notion to the case of the monad $\widetilde\DDD$
on the $2$-category \linebreak $\Fun((\Dt)\op,\Cat)$ of strict $\Cat$-valued
$\Dt$-presheaves (that it, $\DDD$-coalgebras in $\Cat$), we arrive
precisely at the notion of toplax simplicial category. The new top face
operators constitute the structure map for such an algebra. The fact that
the underlying object in $\DDD$-coalgebras is strict is precisely to say
that the restriction to $\Dt$ is strict. The fact that the structure map
is a morphism of strict $\DDD$-coalgebras translates precisely into the
condition that squares that do not involve two consecutive top face
operators must be strictly commutative. The natural simplicial 
transformation $\beta$ 
has components $\beta_n$ which are precisely the beta cells from item (3) in
Definition~\ref{def:toplax}. Equation~\eqref{TTTA} becomes the sequence of equations
\eqref{beta-eq} and Equation~\eqref{trapez} becomes the sequence of equations 
\eqref{beta-eq-bis}. Equations~\eqref{new-eq} and \eqref{new-eq-bis} are
naturality of $\beta$ from \eqref{betaT}.

\begin{remark}\label{Jardine-rmk}
  Except in the special case where all $2$-cells are invertible (which we
  treat in Section~\ref{sec:decomp}), it is not a priori clear that a
  toplax simplicial object is in fact a special case of an oplax
  simplicial object (i.e.~an oplax functor $\simplexcategory\op\to\Cat$).
  The subtlety is that the definition of toplax simplicial object
  provides some non-invertible squares, whereas an oplax simplicial
  object is based on triangles. Each $\beta$-square would correspond to
  two triangles, and the oplax cells sitting in them qua oplax functor
  would be oriented in opposite directions preventing them from being
  composed, as below left
  \[
  \begin{tikzcd}[sep={36pt,between origins}]
  X_{n+2} \ar[dd, "d_{n+1}"'] \ar[rr, "d_{n+2}"] \ar[rrdd]&& X_{n+1} 
  \ar[dd, "d_{n+1}"]  \\
  & \ar[ru, Rightarrow, shorten <=12pt, shorten >=12pt, "\beta"] 
    \ar[ld, Rightarrow, shorten <=10pt, shorten >=10pt, "\omega"'] &\\
  X_{n+1} \ar[rr, "d_{n+1}"'] && X_{n}
  \end{tikzcd}
  \qquad\qquad
  \begin{tikzcd}[sep={36pt,between origins}]
  X_{n+2}  \ar[dd, "d_{n+1}"'] \ar[rr, "d_{n+2}"] \ar[rrdd]&& X_{n+1}  
  \ar[dd, "d_{n+1}"]  \\
  & \ar[ru, Rightarrow, shorten <=12pt, shorten >=12pt, "\beta"] 
   &\\
  X_{n+1}  \ar[rr, "d_{n+1}"'] 
  \ar[ru, Rightarrow, shorten <=10pt, shorten >=10pt, pos=0.8, "\omega^{-1}"]
  && X_{n}
  \end{tikzcd}
  \]
  But in fact one can uniformly define all the $\omega$-cells to be 
  identities. First of all we should choose an image for every 
  composite of generating arrows in $\simplexcategory$:
  each composite corresponds to a simplicial identity, and for all the 
  commuting squares we can choose either way around the square -- they are
  equal. For the only non-commuting square (the $\beta$ squares) we choose as 
  image the composite not involving two consecutive top face operators, 
  thus making the lower triangles in each square an identity $2$-cell.
  In this way we have provided oplax $2$-cells for all triangles in
  $\simplexcategory$. 

  To justify why this is enough to fully define a (normal) oplax functor,
  a coherence theorem is required. Such a coherence result can be
  established based on the supercoherence result of
  Jardine~\cite{JARDINE1991103}, quoted as
  Proposition~\ref{prop:Jardine}. It states that for {\em invertible}
  $2$-cells corresponding to the simplicial identities, certain 17
  families of equations listed among these $2$-cells suffice to obtain a
  fully coherent pseudo-simplicial object. We will invoke Jardine's
  result in the invertible case in Section~\ref{sec:decomp}. In view of
  the subtleties with squares versus triangles, it is not clear that
  Jardine's result can be upgraded to the case of non-invertible
  $2$-cells in full generality, but in the very special case where only
  the $\beta$-cells are not identities, Jardine's proof seems to carry
  over. We do not wish to provide the details, since for the present
  paper it is not logically necessary to know that a toplax simplicial
  object (defined as above in terms of decalage) is actually a special
  case of an oplax functor.
\end{remark}

The following lemma will be useful to establish toplax-ness.

\begin{lemma}\label{reduce-coherence}
  Let $X$ be a pre-toplax simplicial category, meaning that it has all the data
  of a toplax simplicial category, but that the coherence equations have not yet
  been verified. Suppose that all functors $d_0$ are faithful. Then to verify the coherence 
  equations~\eqref{beta-eq}--\eqref{new-eq-bis}, it is enough to check 
  Equation~\eqref{new-eq} for all $n\geq 0$ and $i=0$ and check \eqref{beta-eq} and 
  \eqref{beta-eq-bis} for $n=0$.
\end{lemma}

\begin{proof}
  We shall check each of the remaining equations by first applying the faithful 
  functor $d_0$. We first use a
  standard simplicial argument to show that Equation~\eqref{new-eq} for all $n\geq 0$ and
  $i=0$ implies Equation~\eqref{new-eq} in general. Assume the equation holds for all $i<n$
  and prove by induction that it also holds in the case $n+1$, $j<n+1$, that is $d_j
  \beta_{n+2} = \beta_{n+1} d_j$. If $j=0$ there is nothing to prove. Otherwise it is enough
  to establish the equation after applying $d_0$. But in the resulting equation
  $d_0 d_j \beta_{n+2} = d_0
  \beta_{n+1} d_j$, we can move the $d_0$ to the right (using simplicial identities and the initial assumption). The
  equation then becomes $d_{j-1} \beta_{n+1} d_0 = \beta_{n} d_{j-1} d_0$, which holds by
  the induction hypothesis. The same argument is used to show that the $i=0$ case of 
  \eqref{new-eq-bis} implies all the other cases of \eqref{new-eq-bis}. We now show that
  the $i=0$ case of \eqref{new-eq-bis} follows from the $i=0$ case of \eqref{new-eq}.
  To establish $s_0 \beta_n = \beta_{n+1} s_0$, it is enough to establish 
  $d_0 s_0 \beta_n = d_0 \beta_{n+1} s_0$, since $d_0$ is faithful. On the left, we have
  $d_0 s_0 = \id$, and on the right we can shift $d_0$ right using \eqref{new-eq}, until it 
  reaches $s_0$ and cancels. So we are left with $\beta_n=\beta_n$ which is true.
  
  We now come to the two `substantial' equations, \eqref{beta-eq} and \eqref{beta-eq-bis}.
  To establish the commutativity of the square \eqref{beta-eq} for $n$, it is enough to
  apply $d_0$ and check. But in each of the four sides of the diagram, the $d_0$ can be 
  moved right, using simplicial identities and Equation~\eqref{new-eq}. Doing this yields 
  the $(n-1)$-instance of \eqref{beta-eq} (applied to $d_0$), which we can assume to hold
  by induction (the induction starts at $n=0$, so that case cannot be eliminated).
  The same argument applies to Equation~\eqref{beta-eq-bis}.
\end{proof}

\section{The pita nerve}
\label{sec:pitanerve}

We return to the categories $\fnerve(\catO)_n$, continuing the standing 
assumption that $\catO$ is a \strictlyfactorisable
operadic category.

\medskip

We aim to assemble the categories $\fnerve(\catO)_n$ into a
toplax simplicial category.

There are three arguments involved:

1) 
The categories $\wnerve(\catO)_n$ form a strict simplicial category.

2)
Levelwise we have the reflective subcategories $\fnerve(\catO)_n \subset
\wnerve(\catO)_n$, and we exploit the reflections to obtain 
a pre-simplicial structure on the categories $\fnerve(\catO)_n$
-- by this we mean all the face and degeneracy operators, but not all
the simplicial identities;

3) Finally we exploit the properties of \strictlyfactorisable
categories to show that non-strict part of this pre-simplicial category
enjoys the coherence axioms for being a toplax simplicial category.

\begin{proposition}
  The categories $\wnerve(\catO)_n$, $n\ge 0$ form a strict simplicial 
  category
  $$
\wnerve(\catO) : \simplexcategory\op\to\Cat ,
$$  
  whose face and degeneracy operators we denote by $\dw_n$ and $\sw_n$, 
  respectively.
\end{proposition}

\begin{apology}
  Note that the top face operator $\dw_n$ omits object $T_0$ and that the
  bottom face operator $\dw_0$ omits object $T_n$, which is unfortunate
  from the viewpoint of standard simplicial indexing. We apologise for
  this awkward situation. We have made this choice because it seems more
  important that locally order-preserving chains (which are the main
  interest) end in object
  $T_0$, the reference point for the notion of fibrewise order-preserving
  map, and that this index should always be zero, not a number depending
  on the length of the chain.
\end{apology}

In contrast to $\wnerve(\catO)$, the categories $\fnerve(\catO)_n$ do not
form a strict simplicial object by restriction of simplicial operators. The reason
is that the result of an application of the last simplicial operator
$\dw_n$ to a locally order-preserving chain $T_n\xto{f_n} \ldots \xto{f_1} T_0$ gives
 $T_n\xto{f_n} \ldots \xto{f_2} T_1$, which is not
necessary a locally order-preserving chain.
But we do get a simplicial-object-with-missing-top-face-operators, or 
more formally, a strict presheaf on the subcategory $\Dt\subset\simplexcategory$ of 
last-element-preserving order-preserving maps:
Let
$$
\dd_i = \dw_i:\fnerve(\catO)_{n+1}\to \fnerve(\catO)_n, 
\ \text{and}   \ \ss_i= \sw_i : \fnerve(\catO)_n\to \fnerve(\catO)_{n+1}, \ 0\le i \le n
$$
be the restrictions of the corresponding face and degeneracy operators of
$\wnerve(\catO)$. (Only the top face operator in each degree is missing.)

\begin{lemma}\label{nat-rrr}
  The reflection functors $\rrr_n : \wnerve(\catO)_n \to \fnerve(\catO)_n$
  assemble into a natural transformation of $\Dt$-presheaves.
\end{lemma}

\begin{proof}
  We need to check the commutativity of the squares
   \begin{equation*}
     \centerline{
    \xymatrix@C = +3em{
     \wnerve(\catO)_{n+1} 
     \ar[d]_{\dw_i} \ar[r]^{\rrr_{n+1}}& 
      \fnerve(\catO)_{n+1}
      \ar[d]^{\dd_i} \\
    \wnerve(\catO)_n \ar[r]_{\rrr_n} & \fnerve(\catO)_n
}    
\quad \text{\raisebox{-4.5ex}{and}} \quad
    \xymatrix@C = +3em{
     \wnerve(\catO)_{n+1} 
      \ar[r]^{\rrr_{n+1}}& 
      \fnerve(\catO)_{n+1}
       \\
    \wnerve(\catO)_n \ar[u]^{\sw_i} \ar[r]_{\rrr_n} & \fnerve(\catO)_n 
    \ar[u]_{\ss_i}
}    
\quad \text{\raisebox{-4.5ex}{for  $0\leq i \leq n$.}}
}
\end{equation*}
For the face operators: going right and then down in the square,
the reflection functor $\rrr_{n+1}$ sends a chain
$T_{n+1} \stackrel{f_{n+1}}\to \cdots 
\stackrel{f_2}\to T_1\stackrel{f_1}\to T_0$ to the chain
$$
T'_{n+1} \stackrel{\eta(f_{n+1}/f_1\cdots f_n)}\longrightarrow \cdots \stackrel{\eta(f_2/f_1)}\longrightarrow 
T'_1\stackrel{\eta(f_1)}\longrightarrow T_0 ,
$$
and the $i$th face operator $\dd_i$ then composes at object $T'_{n+1-i}$ to get 
at that step
\begin{equation}\label{rightdown}
T'_{n+2-i} \xto{\eta(f_{n+1-i}/f_1\cdots f_{n-i}) \circ \eta(f_{n-i}/f_1\cdots 
f_{n-i-1})} T'_{n-i} .
\end{equation}
If instead we go down and then right in the square, the effect of the 
face operator $\dw_i$ is to compose at 
object $T_{n+1-i}$ to get (at that step) $f_{n+2-i}\circ f_{n+1-i}$, and then we 
apply $\rrr_n$ to get a long chain of eta maps, where the interesting 
part is
\begin{equation}\label{downright}
T'_{n+2-i} \xto{\eta(f_{n+1-i}\circ f_{n-i} / f_1\cdots 
f_{n-i-1})} T'_{n-i} .
\end{equation}
Elsewhere in the chain it is clear that the two ways around give the 
same result. The apparent discrepancy between \eqref{rightdown} and 
\eqref{downright} is 
resolved by Lemma~\ref{etafunction} item (3), which shows that they are in fact 
equal. 

The argument for the degeneracy operators is similar: starting with a 
chain $T_n \to \cdots \to T_0$, the effect of applying $\rrr_n$ is to 
get a long chain of eta maps, and $\dd_i$ then duplicates the object
$T'_{n-i}$ by inserting an identity map
$$
T'_{n-i} \stackrel{\id}\to T'_{n-i} \,.
$$
The other way around, we first apply $\sw_i$ to have an identity map 
$T_{n-i} \stackrel{\id}\to T_{n-i}$, and then we apply $\rrr_{n+1}$ 
to get a chain of eta maps. The interesting step in that chain is
$$
T'_{n-i} \xto{\eta(\id/f_1\cdots f_{n-i})} T'_{n-i}   \,.
$$
Again the apparent discrepancy between these two results is resolved by 
\linebreak
Lemma~\ref{etafunction}, this time item (5). (Note that for $i=n$ we need
Lemma~\ref{etafunction} item (4) instead.)
\end{proof}

In each simplicial degree $n\geq 0$, we define separately the top face operator
$$
\dd_n: \fnerve(\catO)_n \to \fnerve(\catO)_{n-1} 
$$ 
as the composite
$$
\fnerve(\catO)_n \xto{\iii_n} \wnerve(\catO)_n \xto{\dw_n} \wnerve(\catO)_{n-1}
\xto{\rrr_{n-1}}  \fnerve(\catO)_{n-1}  \,.
$$
For what follows it will be convenient to consider this operator as a
restriction of the endofunctor $\tttt_{n-1}\dw_n:\fnerve(\catO)_n\to\fnerve(\catO)_n$ and
write
$$
\dd_n = \tttt_{n-1}\dw_n \,.
$$
Because the definition of these top face operators
$\dd_n$ involves this extra back-and-forth,
they will not 
satisfy the strict simplicial identities, but we shall see shortly that
we get precisely the required $\beta$-cells

\[
\begin{tikzcd}
\fnerve(\catO)_{n+2} \ar[d, "\dd_{n+1}"'] \ar[r, "\dd_{n+2}"] & 
\fnerve(\catO)_{n+1} \ar[d, "\dd_{n+1}"]  \\
\fnerve(\catO)_{n+1} \ar[r, "\dd_{n+1}"'] \ar[ru, Rightarrow, shorten <=14pt, shorten 
>=14pt, "\beta_n"]& \fnerve(\catO)_n
\end{tikzcd}
\]
to have a toplax simplicial object (see Equation~\eqref{beta} further 
below).

\begin{lemma}\label{strict-part}
  All the other simplicial identities involving top face operators hold 
  strictly:
  $$
  \xymatrixrowsep{2.1em}
  \xymatrixcolsep{2.7em}
  \xymatrix 
  @!R0 
  @!C0
    {
     P_{n+2} 
     \ar[dd]_{d_{i}} \ar[rr]^{d_{n+2}} && 
      P_{n+1}
      \ar[dd]^{d_{i}} \\
      & \text{\footnotesize \rm (i)}
      & \\
    P_{n+1} \ar[rr]_{d_{n+1}} && P_{n}
  }
  \qquad
  \xymatrix  
  @!R0 
  @!C0
    {
     P_{n+2}
      \ar[rr]^{d_{n+2}} && 
      P_{n+1}
       \\
      & \text{\footnotesize \rm (ii)}
      & \\
    P_{n+1} \ar[uu]^{s_{i}} \ar[rr]_{d_{n+1}} && P_{n} \ar[uu]_{s_{i}}
  }
  \qquad
  \xymatrix
  @!R0 
  @!C0
    {
     P_{n+1}
      \ar[rr]^{d_{n+1}} && 
      P_{n}
       \\
      & \text{\footnotesize \rm (iii)}
      & \\
    P_{n} \ar[uu]^{\hspace{-8pt}s_{n}} \ar@/_1.2pc/[rruu]_{\id} & 
    \phantom{P_{n}}
  }
$$
\end{lemma}

\begin{proof}
  For fixed  $n\geq 0$ and 
$i\leq n$, square~(i) is obtained as the composite
   \begin{equation*}
    \xymatrix@C = +3em{
    \fnerve(\catO)_{n+2} \ar[d]_{\dd_i} \ar[r]^{\iii_{n+2}} & 
    \wnerve(\catO)_{n+2} \ar[d]_{\dw_i} 
    \ar[r]^{\dw_{n+2}} & \wnerve(\catO)_{n+1} 
     \ar[d]_{\dw_i} \ar[r]^{\rrr_{n+1}}& 
      \fnerve(\catO)_{n+1}
      \ar[d]^{\dd_i} 
    \\
     \fnerve(\catO)_{n+1} \ar[r]_{\iii_{n+1}}  & \wnerve(\catO)_{n+1} 
     \ar[r]_{\dw_{n+1}}& \wnerve(\catO)_{n} \ar[r]_{\rrr_{n}} & 
     \fnerve(\catO)_{n} \,.
}    
\end{equation*}
Here the first square commutes since $i\leq n$, so we are in the $\Dt$ 
range. The second square is a simplicial identity for $\wnerve(\catO)$ 
(since $i\leq n$), and the third square commutes by 
$\Dt$-naturality of $\rrr$ (cf.~Lemma~\ref{nat-rrr}).

The argument for the sequence of squares (ii) is essentially the same.

For the triangle (iii), the argument is
   \begin{equation*}
    \xymatrix@C = +3em{
    \fnerve(\catO)_{n+1}  \ar[r]^{\iii_{n+1}} & 
    \wnerve(\catO)_{n+1} 
    \ar[r]^{\dw_{n+1}} & \wnerve(\catO)_{n} 
     \ar[r]^{\rrr_{n}}& 
      \fnerve(\catO)_{n} \,.
    \\
     \fnerve(\catO)_{n} \ar[u]^{\ss_{n}} \ar[r]_{\iii_{n}}  & 
     \wnerve(\catO)_{n} \ar[u]^{\sw_{n}} 
     \ar[ru]_{\id}&  & 
}    
\end{equation*}
Here the square commutes because we are within the $\Dt$ range; the 
triangle is a simplicial identity for $\wnerve(\catO)$, and now
$\iii_{n}$ and $\rrr_{n}$ cancel out strictly, so we are left with the 
identity.
\end{proof}

We define $\beta$ to be the natural transformation
  \begin{equation}\label{beta}
  \beta_{n}:\dd_{n+1}\dd_{n+1}\to \dd_{n+1} \dd_{n+2}
  \end{equation}
  given as the composite
  $$
  \tttt_{n}\dw_{n+1}\dw_{n+1}  = \tttt_{n}\dw_{n+1}\dw_{n+2} 
  \xto{ \ \tttt_{n}\dw_{n+1}\unit_{n+1}{\dw_{n+2}} \ }
  \tttt_{n}\dw_{n+1}\tttt_{n+1}\dw_{n+2},
 $$
  where $\unit_{n+1}{\dw_{n+2}}:\dw_{n+2} \to \tttt_{n+1}\dw_{n+2}$ is the unit of
  the monad $\tttt_{n+1}$ precomposed with $\dw_{n+2}$.
  Note that the components of $\beta_{n}$ are fibrewise 
  order-preserving, cf.~\eqref{Wmorphisms}.
  
\begin{theorem}\label{laxcoherence}
  For a \strictlyfactorisable operadic category $\catO$, the sequence of
  categories $\fnerve(\catO)_n$, $n\ge 0$ with the face and degeneracy
  operators and $\beta$-cells as defined above form a toplax simplicial object
  $\fnerve(\catO)$ in $\Cat$.
\end{theorem} 

\begin{proof}
  We have already checked in Lemma~\ref{strict-part} that all the simplicial identities
  other than the $\beta$ cases are strict. We need to check that the $2$-cells $\beta$
  satisfy Equations~\eqref{beta-eq} and \eqref{beta-eq-bis}, as well as
  Equations~\eqref{new-eq} and \eqref{new-eq-bis}. Since we know from Lemma~\ref{df} that
  $d_0$ is a discrete opfibration, and in particular is faithful,
  Lemma~\ref{reduce-coherence} tells us that it is enough to establish
  Equation~\eqref{new-eq} for all cases $n\geq 0$ and $i=0$, and then the $n=0$ case of the
  two `substantial' equations \eqref{beta-eq} and \eqref{beta-eq-bis}.

  Let us first analyse the $\beta$-cells, starting with
  $\beta_0:\dd_{1}\dd_{1}\to \dd_{1} \dd_{2}$. By
  definition, the component $(\beta_0)_a$ on a locally order-preserving chain $a=
  (T_2\xto{f_2} T_1 \xto{f_1}T_0)$ is given by
  $\tttt_0(\dw_1(\unit_1{\dw_2(a)})).$ The monad $\tttt_0$ is the
  identity and $\dw_2(a) = (T_2\xto{f_2} T_1) .$ The unit of the monad
  $\tttt_1$ on this chain is given by the \fop square:
    \begin{equation*}
    \xymatrix@C = +2em{
     T_{2} \ar[d]_{f_2} \ar[r]^{\pi(f_2)}& T'_{2} \ar[d]^{\eta(f_2)} 
 \\
     T_1 \ar[r]_{\id} & T_1   \,.
}    
\end{equation*}  
  Hence $(\beta_0)_a$ is equal to the quasibijection $\pi(f_2):T_2 =
  \dd_1\dd_1(a) \to T_2' = \dd_1\dd_2 (a).$

  Let us also compute the component of $\beta_1:\dd_{2}\dd_{2}\to \dd_{2}
  \dd_{3}$ at a locally order-preserving chain $b= (T_3\xto{f_3}T_2\xto{f_2} T_1
  \xto{f_1} T_0)$. The unit $\unit_2\delta_3(b)$ is presented by the \fop
  diagram
 \begin{equation*}
    \xymatrix@C = +5em{
     T_3
      \ar[d]_{f_3}
       \ar[r]^{\pi(f_2f_3)}& 
      T'_3 \ar[d]^{\eta(f_3/f_2)}
           \\
    T_2 \ar[d]_{f_2}\ar[r]^{\pi(f_2)} & T'_2 \ar[d]^{\eta(f_2)}  
      \\ 
     T_1 \ar[r]_{\id} & T_1    \,.
}    
\end{equation*}
  The result of application of $\delta_2$ to this morphism is the top
  square in this diagram. We now apply $\tttt_1$ to this square and
  obtain a fibrewise order-preserving square which represents
  $(\beta_1)_b$:
\begin{equation}\label{bb1}
    \xymatrix@C = +5em{
     T'_3 \ar[d]_{\eta(f_3)} \ar[r]^{\pi(\pi(f_2)\eta(f_3))}& T''_3 \ar[d]^{\eta(\eta(f_3/f_2))}
      \\
    T_2 \ar[r]_{\pi(f_2)} & T'_2  \,.
}    
\end{equation}

  In general, with reference to Diagram~\eqref{verticalcomposition} the
  transformation $\beta_{n-2}$ on a locally order-preserving chain $ a =
  (T_{n}\stackrel{f_{n}}{\to} \cdots \stackrel{f_1}{\to} T_0)$ is given
  by the following \fop diagram:
\begin{equation}\label{beta-fop}
    \xymatrix@C = +2.5em{
     T'_{n} \ar[d]_{\eta(f_{n}/f_{3}\cdots f_{n-1})} 
     \ar[rr]^{\pi(\pi(f_2)\eta(f_3\cdots f_n))}
     & &
     T''_{n} \ar[d]^{\eta(\eta(f_n/f_{2}\cdots 
     f_{n-1})/\eta(f_{3}\cdots f_{n-1}/f_2))} 
           \\ 
           \vdots \ar[d] & \vdots & \vdots \ar[d] \\
 T_3'   \ar[d]_{\eta(f_3)}  \ar[rr]^{\pi(\pi(f_2)\eta(f_3))}  & & 
 T_3''
   \ar[d]^{\eta(\eta(f_3/f_2))}
      \\ 
     T_2 \ar[rr]_{\pi(f_2)} &&  \phantom{\,,}T'_2   \,,
}    
\end{equation}   
  where the right-hand side of the diagram is calculated using the
  relations of Lemma~\ref{etafunction}.

  Having understood the components of $\beta$, it is easy to establish
  the $i=0$ case of Equation~\eqref{new-eq}, which says
  $$
  d_0 \beta_{n+1} = \beta_n d_0.
  $$
  This equation holds because it is the same first to delete the beginning of the
  chain and then apply the $\beta$-construction, or first applying $\beta$ and then
  delete the beginning of the chain.
  
  It now remains to establish the two `substantial' equations \eqref{beta-eq} and 
  \eqref{beta-eq-bis}, and it is now enough to deal with the $n=0$ case.
   Equation~\eqref{beta-eq} is the
  commutative diagram
  \begin{equation}\label{betacoherence1}
    \begin{tikzcd}[row sep={9ex,between origins}, column sep={7em,between origins}]
    \dd_{1} \dd_{1}\dd_{1}\ar[r, "\id"] \ar[d, "\beta_{0}\dd_{1}"'] 
    & \dd_{1}\dd_{1}\dd_{2} \ar[r, "\beta_{0}\dd_{2}"]
    & \dd_{1} \dd_{2} \dd_{2} \ar[d, "\dd_{1}\beta_{1}"]
    \\
    \dd_{1}\dd_{2} \dd_{1} \ar[r, "\id"']
    & \dd_{1}\dd_{1}\dd_{3}  \ar[r, "\beta_{0}\dd_{3}"']
    & \dd_{1}\dd_{2}\dd_{3} \,.
    \end{tikzcd}
\end{equation}
  which amounts to
  the equality
  \begin{equation}\label{coh}
	(\beta_{0}  \dd_{3})\circ( \beta_{0}  \dd_{1}) =  (\dd_{1}   
    \beta_{1})\circ(\beta_{0}  \dd_{2}). 
  \end{equation} 
  We start from
  $(\beta_{0} \dd_{3})( \beta_{0} \dd_{1})$ computed at $b=
  (T_3\xto{f_3}T_2\xto{f_2} T_1 \xto{f_1} T_0).$ The transformation
  $\beta_0\dd_1$ on $b$ is the quasibijection $\pi(f_2f_3).$ The
  transformation $\beta_0\dd_3$ on $b$ is equal to $\beta_0$ on $\dd_3(b)=
  (T_2'\xto{\eta(f_3/f_2)} T_2'\xto{\eta(f_2)}T_1)$ and therefore is equal
  to the quasibijection $\pi(\eta(f_3/f_2))$. Thus the left-hand side of
  the equation is $\pi(\eta(f_3/f_2))\pi(f_2f_3)$. To compute the
  right-hand side $(\dd_{1} \beta_{1})(\beta_{0} \dd_2)$ on $b$ we first see
  that $\beta_0\dd_2$ is given by the quasibijection $\pi(f_3)$, whereas
  $\dd_1\beta_1$ will be equal to $\pi(\pi(f_2)\eta(f_3))$ according to
  \eqref{bb1}. So in the end the equation we need to establish is
\begin{equation}\label{first1}
  \pi(\eta(f_3/f_2))\pi(f_2f_3) = \pi( \pi(f_2)\eta(f_3))\pi(f_3) \,.
\end{equation}
  To establish this equation, we apply Proposition~\ref{fundamental} to
  the commutative \fop square
\begin{equation*}
    \xymatrix@C = +4em{
     T_3 \ar[d]_{f_3} \ar[r]^{\pi(f_2f_3)}& T'_3 \ar[d]^{\eta(f_3/f_2)}
      \\
    T_2 \ar[r]_{\pi(f_2)} & T'_2
}    
\end{equation*}
to get 
 \begin{equation*}
    \xymatrix@C = +7em{
     T_3\ar@/^-3.3ex/[dd]_{f_3}
      \ar[d]^{\pi(f_3)} \ar[r]^{\pi(f_2f_3)}& T_3' \ar[d]_{\pi(\eta(f_3/f_2)}
      \ar@/^3.3ex/[dd]^{\eta(f_3/f_2)} 
           \\
    T''_3 \ar[d]^{\eta(f_3)} \ar[r]^{\pi(\pi(f_2)\eta(f_3))} & T'''_3 \ar[d]_{\eta(\eta(f_3/f_2)}  
      \\ 
     T_2 \ar[r]_{\pi(f_2)} & T'_2   \,.
}    
\end{equation*}
  But commutativity of the top square is now precisely the desired
  Equation~\eqref{first1}.
  
  Finally we check Equation~\eqref{beta-eq-bis} for $n=0$, which states 
  $$
  \beta_0 s_1 = \id .
  $$
  The input to this equation is
  a (locally) 
  order-preserving $1$-chain
  $T_1\to T_0$. The degeneracy operator sends it to $T_1 \to T_0
  \stackrel{\id}\to T_0$. The claim is that the component of the
  $\beta$-cell at this element is trivial. But this is true, because
  $d_2$ just omits the last copy of $T_0$, and since in this case the
  remaining $T_1\to T_0$ is already order-preserving, the reflection does
  nothing to it. So $\beta$ acts as the identity on such degenerate chains.
\end{proof}

\section{Decomposition spaces}
\label{sec:decomp}

We finish with some further treatment of the important case where the strictly
factorisable operadic $\catO$ has the property that all quasibijections
are invertible. We show that in this case the pita nerve is a 
decomposition space~\cite{Galvez-Kock-Tonks:1512.07573} (a.k.a.~$2$-Segal 
space~\cite{Dyckerhoff-Kapranov:1212.3563}).

To show first of all that the pita nerve is a fully coherent pseudo-simplicial 
groupoid, we invoke Jardine's supercoherence result.

\begin{definition}
  We use the word {\em pre-simplicial object} in a $2$-category $\KK$ for the 
  data of 
  \begin{enumerate}
    \item An object $X_n$ for each $n\geq 0$;
    
    \item Face and degeneracy operators $d_i$ and $s_j$ like in a 
    simplicial object;
  
    \item $2$-cells replacing all of the usual simplicial identities:    
    
  $$d_j d_i \to d_i d_{j+1}, i\le j,$$
  $$d_i s_j \to s_{j-1}d_i, i<j,$$
  $$d_i s_j \to s_{j}d_{i-1}, i>j+1,$$   
  $$\id\to d_i s_i,$$
  $$\id\to d_{i+1}s_i,$$
  $$s_i s_j \to s_{j+1} s_i , i\le j.$$
  \end{enumerate}
\end{definition}

\begin{proposition}
  [Supercoherence, Jardine~\cite{JARDINE1991103}]
  \label{prop:Jardine}
  A pre-simplicial object becomes a 
  fully coherent pseudo-simplicial object (that is, a pseudo-functor 
  $\simplexcategory\op\to\KK$) provided all the $2$-cells are invertible
  and satisfy a specific list of 17 families of equations (listed as 
  Equations~(1.4.1)--(1.4.17) in \cite{JARDINE1991103}).
\end{proposition}

\begin{proposition}
  For a strictly factorisable operadic category $\catO$ in which all 
  quasibijections are invertible, the pita nerve is a fully coherent 
  pseudo-simplicial groupoid.
\end{proposition}

\begin{proof}
  The data of a pre-simplicial object has already been provided in
  Section~\ref{sec:pitanerve}: all the $2$-cells are identities except the
  $\beta$-cells. Note that since the components of $\beta$ are now invertible,
  the pita nerve takes values in groupoids rather than categories. Since most of
  the $2$-cell data is trivial, 14 out of the 17 families of equations are
  satisfied trivially (by inspection). There are only three families of
  equations left in Jardine's list, and they are precisely
  Equations~\eqref{beta-eq} and \eqref{new-eq} (corresponding together to
  Jardine's 1.4.1) and Equations~\eqref{beta-eq-bis} (Jardine's 1.4.4), and
  finally Equations~\eqref{new-eq-bis} (Jardine's 1.4.2). We already checked
  \eqref{beta-eq}--\eqref{new-eq}, and \eqref{new-eq-bis} follows as a
  consequence since we are in the situation of Lemma~\ref{reduce-coherence}.
\end{proof}

This pseudo-simplicial groupoid is generally not Segal (except in the special
case where all morphisms in $\catO$ are order-preserving, which is to say that
the cardinality functor factors through the subcategory $\simplexcategory_+
\subset \Fin$). Indeed, the Segal condition says that the map $P_2 \to
P_1\times_{P_0} P_1$ is an equivalence. But for a composable pair of
order-preserving morphisms $T_2 \stackrel{f}\to T_1 \stackrel{g}\to T_0$ (an
element in $P_1\times_{P_0} P_1$), there are generally many locally
order-preserving chains $p\in P_2$ whose $\dd_2$ and $\dd_0$ give $f$ and $g$.
For example, in addition to the order-preserving chain consisting of $f$ and $g$
themselves, one can insert a fibrewise order-preserving quasibijection to get
other locally order-preserving chains with the same image in $P_1\times_{P_0}
P_1$.

Seeing thus a kind of multi-valued composition, it is natural to ask if
the pita nerve $P$ (of a strictly factorisable operadic category $\catO$ 
in which all quasi\-bijections are invertible) is a decomposition 
space~\cite{Galvez-Kock-Tonks:1512.07573} 
(a.k.a.~$2$-Segal space~\cite{Dyckerhoff-Kapranov:1212.3563}).
This is indeed the case:

\begin{theorem}\label{thm:decomp}
  For any strictly factorisable operadic category $\catO$ 
in which all quasibijections are invertible, the pita nerve 
$P=\fnerve(\catO)$ is a decomposition space.
\end{theorem}
\begin{proof}
  Since clearly the upper decalage $\DDD(P)$ is a Segal space, we are 
  left with checking that the following squares are pullbacks for all 
  $n\geq 0$:
  \[
  \begin{tikzcd}
  P_{n+3} \ar[d, "\dd_{n+3}"'] \ar[r, "\dd_1"] & P_{n+2} \ar[d, 
  "\dd_{n+2}"]  \\
  P_{n+2} \ar[r, "\dd_1"'] & P_{n+1}
  \end{tikzcd}
  \]
  We do the case $n=0$. All the higher cases are analogous, only with 
  longer chains. For each element $p\in P_2$ in the upper right corner 
  of the diagram we compare the fibre of $\dd_1$ over $p$ with the 
  fibre of $\dd_1$ over $\dd_2(p)$, and show that they are equivalent 
  groupoids. Let $p$ be the locally order-preserving chain 
  $$
  T_3 \stackrel{f}\to T_1 \to T_0 \,.
  $$
  The $\dd_1$-fibre is the groupoid  $G_1$ of all factorisations of $f$
  and fibrewise order-preserving quasibijections between them at the 
  middle object. More precisely this groupoid has objects as the columns
  in the next diagram, and morphisms given by the horizontal map $\sigma$, 
  forming altogether a \fop diagram:
  \begin{equation*}\label{C}
  \begin{tikzcd}
  T_3 \ar[d, "f_3"] \ar[r, "\id"] \ar[dd, bend right, "f"'] & T_3 \ar[d, 
  "\tilde f_3"'] \ar[dd, bend left, "f"]  \\
  T_2 \ar[d, "f_2"] \ar[r, "\sigma"] & \tilde T_2 \ar[d, "\tilde f_2"']  \\
  T_1 \ar[d] \ar[r, "\id"] & T_1 \ar[d]  \\
  T_0  \ar[r, "\id"] & T_0
  \end{tikzcd}
  \end{equation*}
  The other groupoid $G_2$ in play is the fibre over $\dd_2(p)$. Recall that
  $\dd(p)$ is the locally order-preserving chain
  $$
  T_3' \stackrel{\eta(f)}\longrightarrow T_1 ,
  $$
  so the $\dd_1$-fibre over it is the groupoid of factorisations
  \begin{equation}\label{C2}
  \begin{tikzcd}
  T_3' \ar[d, "h"] \ar[r, "\id"] \ar[dd, bend right, "\eta(f)"'] & T_3' \ar[d, 
  "\tilde h"'] \ar[dd, bend left, "\eta(f)"]  \\
  T_2 \ar[d, "e"] \ar[r, "w"] & T'_2 \ar[d, "\tilde e"']  \\
  T_1 \ar[r, "\id"] & T_1 
  \end{tikzcd}
  \end{equation}
  Checking that the two groupoids are equivalent boils down to the following lemma,
  which we separate out since it clarifies the connections between the decomposition-space
  property and strict factorisability.
\end{proof}

The following lemma can be seen as a converse to Proposition~\ref{hitrayadiagramma}.
Fix a morphism $f$ in a strictly factorisable operadic category.
Let $C_1$ be the category whose objects are factorisations $f = f_2
  \circ f_3$ and whose morphisms are \fop diagrams
  \begin{equation}\label{C1}
  \begin{tikzcd}
  T_3 
  \ar[d, "f_3"] 
  \ar[r, "\id"] 
  \ar[dd, bend right, "f"'] 
  & T_3 
  \ar[d, "\tilde f_3"'] 
  \ar[dd, bend left, "f"]
  \\
  T_2 
  \ar[d, "f_2"] 
  \ar[r, "\sigma"] 
  & 
  \tilde{T}_2 
  \ar[d, "\tilde f_2"']
  \\
T_1
  \ar[r, "\id"] 
  & T_1
  \end{tikzcd}
  \end{equation}
Let $C_2$ be the category whose objects are factorisations $\eta(f)= e\circ h$
  with order-preserving $e$ and whose morphisms are diagrams (\ref{C2}). Note that
  $C_2$ is actually discrete. Indeed, since $w$ is a
  \fop quasibijection and $e$ and $\tilde{e}$ are both order-preserving, it follows that
  if $w$ exists it must be an identity.

\begin{lemma}
  There is an adjunction 
  \[
  \begin{tikzcd}
  C_1 \ar[r, shift left=2, "G"] \ar[r, phantom, "\top"{font=\tiny}] & C_2 ,
  \ar[l, shift left=2, "F"]
  \end{tikzcd}
  \]
  which is an equivalence in the case where all quasibijections are invertible.
\end{lemma}

\begin{proof}
  We construct a functor $G:C_1\to C_2$ by application of
  Proposition~\ref{hitrayadiagramma}(2) to a factorisation $f=f_2f_3.$ 
  We obtain a commutative diagram
  \begin{equation}\label{etafunctorialC1}
    \xymatrix@C = +1em@R = +1em{
    T_3 \ar@/^-3.3ex/[ddrrr]_{f} \ar[drr]^{\hskip1em\pi(f)}     \ar[rrrrrr]^{f_3}& & & & & & T_2  \ar@/^3.3ex/[ddlll]^{f_2} \ar[dll]_{\hskip-0.7em\pi(f_2)} 
     \\
     & & T_3'      \ar@{-->}[rr]^{\eta({f_3}/{f_2})} \ar[dr]_(0.2){\eta(f)} & & \tilde T_2 \ar[dl]^(0.2){\hskip-0.7em\eta(f_2)} & &
      \\ 
     & & &T_1& & &
    }
\end{equation}
  Observe that the inner triangle here
  is an object
  of the category $C_2$, and we define $G(f_2\circ f_3)$ to be this 
  factorisation $\eta(f)=
  \eta(f_2)\circ\eta(f_3/f_2)$. Functoriality of this assignment
  follows directly from the discreteness of $C_2$ (or from 
  Proposition~\ref{fundamental}(1)).
  
  To construct a left adjoint to $G$, we consider 
  a given factorisation $\eta(f)= e\circ h$
  and assign to it the factorisation 
  $$F(e\circ h):= e\circ(h\pi(f)).$$
  (This composes to $f$ as required: $e\circ(h\pi(f)) = \eta(f) \circ \pi(f) = 
  f$.)
  We have
  $$
  GF(e\circ h) = G(e\circ(h\pi(f)))=\eta(e)\circ\eta(h\pi(f)/e) = e\circ h,
  $$
  because $e$ is order-preserving (and hence $e=\eta(e)$) and because
$$
\eta(h\pi(f)/e) = \eta(h/e)\eta(\pi(f)/eh) = \eta(h/e)\eta(\pi(f)/\eta(f)) = h \cdot id = h
$$
by Lemma~\ref{etafunction}(3,6)).
This shows that the counit of the adjunction is an identity.
  
 On the other hand,
  $$FG(f_2\circ f_3)= F(\eta(f_2)\circ \eta(f_3/f_2)) =
  \eta(f_2)\circ (\eta(f_3/f_2)\pi(f))$$ and
     the unit is given by the diagram
  \begin{equation}\label{C3}
  \begin{tikzcd}  T_3 \ar[d, "f_3"] \ar[r, "\id"] \ar[dd, bend right, "f"'] & T_3 \ar[d, 
  "\eta(f_3/f_2)\pi(f)"] \ar[dd, bend left, "f"]  \\
  T_2 \ar[d, "f_2"] \ar[r, "\pi(f_2)"] & \tilde T_2 \ar[d, " \eta(f_2)"']  \\
  T_1
   \ar[r, "\id"] & T_1
  \end{tikzcd}
  \end{equation}
  which commutes by
  Diagram~\eqref{etafunctorialC1}. 
  We omit
  the straightforward check of naturality.
  Note that $\pi(f_2)$ is a quasibijection, so if we know that all quasibijections 
  are invertible, then the unit is invertible, and altogether the adjunction is an 
  equivalence in that case.
\end{proof}

One reason to be interested in decomposition spaces is that they have 
incidence coalgebras~\cite{Galvez-Kock-Tonks:1512.07573}, 
\cite{Galvez-Kock-Tonks:1612.09225}. It follows from  Theorem~\ref{thm:decomp} that
any strictly factorisable operadic category with invertible 
quasibijections has associated to it a new coalgebra. 

Recall from \cite{Galvez-Kock-Tonks:1512.07573} that the incidence
coalgebra of a decomposition space $X$ has basis given by iso-classes
of elements in $X_1$, and that the comultiplication is given on  a basis element
$f\in X_1$ as
$$
\Delta(f) = \sum_{p \in X_2 \,\mid\, \dd_1(p)=f} \dd_2(p) \otimes \dd_0(p) .
$$
In the present case of the pita nerve $P$, the formula thus gives,
for any order-preserving morphism $f$:
$$
\Delta(f) = \sum_{f=e\circ h} \eta(h) \otimes e \,,
$$
where in the sum we require $e$ to be order-preserving, but $h$ is not 
required to be so.

It is beyond the scope of this paper to investigate the significance of
this construction, but as an illustration we just work out one basic example.

\begin{blanko}{Example: finite sets and surjections.}
  We consider the example of $\Fin_{\mathrm{surj}}$. This is a strictly factorisable
  operadic category, and the quasibijections are just the bijections. The comultiplication
  is now very explicit, and it is easy to list all the factorisations of a given
  order-preserving surjection. Furthermore, because of the compatibility of ordinal sum with
  pita factorisation in $\Fin$, the coalgebra actually becomes a bialgebra (see
  \cite{Galvez-Kock-Tonks:1512.07573}). To describe the comultiplication, it is thus enough
  to give its effect on connected surjections, meaning surjections of the form $n\to 1$. In
  the left-hand tensor factor of the formula for comultiplication we have first a general
  surjection $h$, but after replacing it with $\eta(h)$ we get an order-preserving
  surjection, and then it splits uniquely as the ordinal sum of connected surjections. So
  altogether everything can be described in terms of connected surjections only. Let $A_n$
  denote the isoclass of the surjection $n\to 1$. Then
\begin{align*}
\Delta(A_1) &= A_1 \otimes A_1 \\
\Delta(A_2) &= 2! \cdot A_1^2 \otimes A_2 + A_2 \otimes A_1 \\
\Delta(A_3) &= 3! \cdot A_1^3\otimes A_3 +  2! \cdot 3  A_2 A_1\otimes A_2+ A_3 
\otimes A_1 \\
\Delta(A_4) &= 4!\cdot A_1^4\otimes A_4 + 3! \cdot 6 A_2A_1^2 \otimes A_3 +
2!\cdot ( 4 A_3 A_1 + 3 A_2^2)\otimes A_2 + A_4\otimes A_1
\end{align*}
The general formula is
$$
\Delta(A_n) = k! \cdot B(n,k)(A_1,A_2,A_3,\ldots) \otimes A_k
$$
where $B(n,k)(A_1,A_2,\ldots)$ are
the Bell polynomials counting partitions of $n$ into $k$ blocks of sizes
specified by the vector $(A_1, A_2,\ldots)$.

  This is very similar to the Fa\`a di Bruno bialgebra, but the factors $k!$ make it
  different from the usual Fa\`a di Bruno bialgebra (which has the Bell polynomials without
  the factorials), and it is not an obvious base change of it either. Recall that the usual
  Fa\`a di Bruno bialgebra arises as the incidence bialgebra of the fat nerve of the
  category $\Fin_{\mathrm{surj}}$ without regard to the operadic-category structure, and
  that the Dynkin--Fa\`a di Bruno bialgebra~\cite{MuntheKaas:BIT95} arises as the incidence
  bialgebra of the category of order-preserving surjections (without regard to
  operadic-category structure). Over the rational numbers, these two are isomorphic by the
  base change that sends $A_n$ to $A_n/n!$. The factorials in the new example enter in a
  different way, intertwined with the summation. (See \cite{Kock-Weber:1609.03276} for these
  two Fa\`a di Bruno bialgebras as well as further generalisations from an operadic
  perspective.)
\end{blanko}


\bigskip

\footnotesize

\noindent
Address: {Department of Mathematics, HSE University, 119048 Usacheva str., 6, Moscow, Russia}

\noindent
Email addresses: \texttt{bataninmichael@gmail.com, mbatanin@hse.ru} 

\medskip

\noindent
Address: {Universitat Aut\`onoma de Barcelona, University of Copenhagen, 
and Centre de Recerca Matem\`atica}

\noindent
Email address: \texttt{joachim.kock@uab.cat} 

\medskip

\noindent
Address: {Macquarie University, Sydney}

\noindent
Email address: \texttt{mark.weber.math@googlemail.com} 

\vfill

\hrule

  \noindent Grant info: We acknowledge support from grant
  No.~10.46540/3103-00099B from the Independent Research Fund Denmark,
  grant PID2024-158573NB-I00 (AEI/FEDER, UE) of Spain and grant
  2021-SGR-1015 of Catalonia, as well as the Severo Ochoa and Mar\'ia de
  Maeztu Program for Centers and Units of Excellence in R\&D grant number
  CEX2020-001084-M and the Danish National Research Foundation
  through the Copenhagen Centre for Geometry and Topology (DNRF151).

\end{document}